\documentclass[titlepage]{article}
\usepackage[utf8]{inputenc}
\usepackage[affil-it]{authblk}

\usepackage{amsmath,bm}
\usepackage{amsfonts}
\usepackage{amssymb}
\usepackage{hyperref}
\usepackage{mathrsfs}
\usepackage[margin=1in]{geometry}
\usepackage[inline]{enumitem}
\usepackage{caption}
\usepackage{xcolor}
\usepackage{wrapfig}
\usepackage{exscale,relsize}

\usepackage{tikz}
\usetikzlibrary{matrix,calc,arrows}

\small\normalsize

\newtheorem{theorem}{Theorem}[section]
\newtheorem{lemma}[theorem]{Lemma}
\newtheorem{proposition}[theorem]{Proposition}
\newtheorem{corollary}[theorem]{Corollary}
\newtheorem{claim}{Claim}
\newtheorem{conjecture}[theorem]{Conjecture}
\newtheorem{observation}[theorem]{Observation}

\newenvironment{proof}[1][Proof:]{\begin{trivlist}
\item[\hskip \labelsep {\bfseries #1}]}{\end{trivlist}}
\newenvironment{definition}[1][Definition:]{\begin{trivlist}
\item[\hskip \labelsep {\bfseries #1}]}{\end{trivlist}}
 \newcounter{nona}[theorem]
\newcounter{nonanona}[theorem]    
\renewcommand{\thenona}{\Alph{nona}}

\newenvironment{noname}{\begin{trivlist}\item[]\refstepcounter{nona}%
        {\bf (\thenona)\ \ \ }\nobreak\noindent\sl\ignorespaces}{%
        \ifvmode\smallskip\fi\end{trivlist}}

\renewenvironment{proof}[1][Proof:]
{\begin{trivlist}
\item[\hskip \labelsep {\bfseries #1}]} {\qquad\hspace*{\fill}$\square$\end{trivlist}}

\newcommand{\bB}{\mathcal{B}}
\newcommand{\B}{\mathfrak{B}}

\newcommand{\F}{\mathcal{F}}
\newcommand{\frakF}{\mathfrak{F}}

\newcommand{\C}{\mathcal{C}}

\newcommand{\T}{\mathcal{T}}

\renewcommand{\epsilon}{\varepsilon}
\renewcommand{\tilde}{\widetilde}

\newcommand\numberthis{\addtocounter{equation}{1}\tag{\theequation}}

\title{Rota's Basis Conjecture for Matroids with Density Close to One }

\author[1]{Sean McGuinness}

\affil[1] {Thompson Rivers University}

\begin{document}

\maketitle
\begin{abstract}
Rota's basis conjecture (RBC) states that given a collection $\bB$ of $n$ bases in a matroid $M$ of rank $n$, one can always find $n$ disjoint rainbow bases with respect to $\bB$.
We show that if $M$ is a matroid having $n + k$ elements, then one can construct $n - k^3$ disjoint rainbow bases.

\vspace{.2in}


\end{abstract}

\section{Introduction}

For basic concepts and notation pertaining to matroids, we follow Oxley \cite{Oxley}, and for graphs, we shall follow Bondy and Murty \cite{Bondy&Murty}. Let $M$ be a matroid of rank $n$. A {\bf base sequence} of $M$ is an $m$-tuple $\bB = (B_1, \ldots B_m) \in \B(M)^m$ of bases of $M$, where for each $i \in \{1,2, \ldots, m\}$, we think of the base $B_i$ as ``coloured'' with colour $i$. A {\bf rainbow base} ({\bf RB}) with respect to $\bB$ is a base of $M$ all of whose elements have different colours. Two rainbow bases with respect to $\bB$ are said to be {\bf disjoint} if for each colour $c$, if both bases have elements of colour $c$, then these elements are distinct. We let $t_M(\bB)$ denote the cardinality of a largest set of disjoint rainbow bases with respect to~$\bB$, where the subscript is dropped when $M$ is implicit.
In 1989, Rota made the following conjecture, first communicated in \cite{HuangRota}:

\begin{conjecture}[Rota's Basis Conjecture (RBC)]\label{conj_rbc}
Let $M$ be a matroid of rank $n$, and let $\bB$ be a base sequence of $M$ consisting of $n$ bases. Then $t(\bB)=n$.
\end{conjecture}

Due to the work of Drisko \cite{Drisko}, Glynn \cite{Glynn}, and Onn \cite{Onn} on the Alon-Tarsi conjecture \cite{Alon&Tarsi}, RBC is known to be true for $\mathbb{F}$-representable matroids of rank $p \pm 1$, where $\mathbb{F}$ is a field of characteristic $0$, and $p$ is an odd prime. See \cite{Friedman&McGuinness} for an overview of these results.
Chan \cite{Chan} and Cheung \cite{Cheung} proved RBC for matroids of rank $3$ and $4$, respectively. In \cite{Geelen&Humphries}, it was shown that RBC is true for sparse paving matroids and Wild \cite{Wild} proved the conjecture for strongly base-orderable matroids.  

While RBC has attracted a significant amount of attention (see the polymath project \cite{Chow}) and a number of different ideas have yielded different results, there are still very few classes of matroids for which it is known to be true.  Even attempts to find lower bounds on the number of disjoint rainbow bases has been challenging.   The bound established by Geelen and Webb \cite{Geelen&Webb}, who proved that $t(\bB) \ge \sqrt{n-1}$, was the only bound known until recently when it was improved by 
Dong and Geelen \cite{Dong&Geelen}, who showed that $t(\bB) \ge \frac n{7\log n}$.   Around the same time, a significant breakthrough was made on RBC by by Buci\'c et al. \cite{Bucic} who showed that $t(\bB) \geq (1/2-o(1))n$, provided $n$ is large enough.    
While a proof for RBC seems out of reach at the moment, a natural question raised in \cite{Bucic} asks whether one can improve the $(\frac 12 - o(1))n$ lower bound to say $(1 - o(1))n$.  This is a very challenging problem given that even finding a constant $c> \frac 12$ for which $t(\bB) \geq cn$ would be a significant step.   As an example of when the $n-o(n)$ bound can be achieved,
 it was shown in \cite{Friedman&McGuinness} that if $M$ is a rank-$n$ matroid having girth at least $n-o(\sqrt{n})$, and no element of $M$ belongs to more than $o(\sqrt{n})$ bases in $\bB$, then one can find at least $n - o(n)$ disjoint rainbow bases with respect to $\bB$.   We mention also an interesting recent result in \cite{Pokrovskiy} where it is shown that one can find at least $n-o(n)$ disjoint rainbow independent sets of size at least $n-o(n).$ 

At this point in time, RBC remains unresolved even for matroids having rank $n$ and a small number of elements.   None of the methods introduced thus far apply to rank-$n$ matroids with $n+k$ elements, where $k$ is fixed (for example, $k=5$). This is true even for graphic matroids.
Thus it is a natural question to ask, if for fixed $k$, is RBC true for rank-$n$ matroids having $n+k$ elements?  Can one find even $n - o(n)$ disjoint rainbow bases in this case? 
The main result of this paper is the following:

\begin{theorem}
Let $\bB$ be a sequence of $n$ bases in a matroid of rank $n$ having $n+k$ elements.  Then $t(\bB) \ge n-k^3.$
\label{the-main}
\end{theorem}

As a consequence of the above theorem we obtain:

\begin{corollary}
Let $\bB$ be a sequence of $n$ bases in a matroid of rank $n$ having $n+o(n^{\frac 13})$ elements.
Then $t(\bB) \ge n - o(n).$\label{cor-main}
\end{corollary}

%
The method of proof of the main theorem is entirely new.    In constructing disjoint rainbow bases,  we ensure that at each step that a certain function defined on the flats of a rank-$k$ matroid satisfies certain inequalities.  When such inequalities are satisfied, one can use the matroid intersection theorem (see \cite{Oxley}) to guarantee the existence of the next rainbow base.  In essence, the rainbow bases are constructed one-at-a-time largely by the special selection of matchings in a bipartite graph.  Below we give a more detailed sketch.   

\subsection{Overview of the proof}

Let $M$ be a binary matroid of rank $n$ where $\varepsilon(M)= |E(M)| = n + k.$  
Let $\bB = ( B_1, \dots ,B_n )$ be a sequence of coloured bases of $M$.   Let $m=n+k$ and let $E(M) = \{ e_1, \dots ,e_{m} \}.$   

We shall define a bipartite graph $G = G(\bB)$ having bipartition $(U,V)$ where $U = \{ u_1, \dots ,u_n \}$, each vertex $u_i$ representing $B_i$, and $V = E(M)$.  We let $V = \{ v_1, \dots ,v_m \}$ where $v_i = e_i,\ i = 1, \dots ,m.$ For all $i,j$, $u_i$ and $v_j$ are joined by an edge in $G$ if and only if $e_j \in B_i.$
Note that $d_G(u_i) = n,\ i = 1, \dots ,n$ and $d_G(v_i) \le n,\ i = 1, \dots ,n.$
We colour the edges of $G$ with colours $1,2, \dots ,n$ such that, for $i = 1, \dots ,n$, the edges incident to $u_i$ receive colour $i.$ A rainbow basis $B$ in $M$ corresponds to a {\it matching} $W$ in $G$ which saturates all vertices in $U$ and those elements in $V$ saturated by $W$ correspond to $B$.  

Suppose that $G$ can be partitioned into $n$ matchings $W_1, \dots ,W_n$, where, for all $i$, $W_i$ corresponds to a rainbow basis of $M$.    For $i = 1, \dots ,n,$ let $G_i = G - W_1 - \dots - W_i. $   
We define the {\it deficit} of a vertex $v\in V$ in $G_i$ to be $\delta_i(v) = n-i - d_{G_i}(v).$ Such a decomposition of $G$ into matchings $W_1, \dots ,W_n$ necessitates that the graph $G_i$ satisfy certain conditions.  First, we must have that for all $v\in V,$ $\delta_i(v) \ge 0.$  For if $\delta_i(v) <0$ for some vertex $v\in V,$ then $d_{G_i}(v) > n-i.$  However, this is impossible if the edges incident to $v$ in $G_i$ belong to the matchings $W_{i+1}, \dots , W_n.$   In addition, it can be shown that the deficit of a set cannot be too large.  More precisely, for a set $A \subseteq E$ let $\delta_i(A) = \delta_i(V_A).$  It can be shown that for every subset $A \subseteq E,$  $\delta_i(A) \le r_N(A)(n-i)$ where $N$ is a matroid of rank $k$ derived from $M.$  We refer to these as {\it deficit inequalities}.

Suppose now that we have only constructed disjoint matchings $W_1, \dots ,W_i$, each corresponding to a rainbow basis.   A  key observation is that if $\delta_i$ satisfies a deficit inequality for each $A\subseteq E,$ then the matroid intersection theorem implies that there exists a matching $W_{i+1}$ in $G_i$ corresponding to a rainbow basis in $M.$   Thus to guarantee that $G_i$ has such a matching it suffices to choose the previous matchings $W_1, \dots ,W_i$ so that $\delta_i$ satisfies these deficit inequalities.  This is a daunting task. However, it turns out that if the deficit inequalities hold for $\delta_{i-1}$ and $n-i$ is large enough, then
one can choose $W_i$ carefully so that the deficit inequalities will hold for $\delta_i$.   Thus to construct disjoint rainbow bases, it suffices to show that at each step, one can first choose $W_i$ corresponding to a rainbow basis, and then alter $W_i$ so that the deficit inequalities hold for $\delta_i.$   We alter $W_i$ using sequences of {\it alternating paths} in $G_{i-1}$ where one swaps the matching edges on the paths with the non-matching edges.
The major challenge is to alter the matchings so that the resulting matching corresponds to a rainbow basis.  Secondly, the new matching must improve on the previous matching by reducing the number of deficit inequalities which are unsatisfied.  We show that this can be achieved by constructing a {\it path-chain} of alternating paths.  The existence of such path-chains will be dependent on $n-i$ being much larger than $k$ (more specifically, $n- i \ge k^3$).  After each step $i$, the construction of the next matching $W_{i+1}$ really only depends on the judicious selection of $W_i$ and as such one can view the rainbow bases as being constructed one-at-a-time.  

\section{Notation}
For a positive integer $k$, let $[k]$ denote the set $\{ 1, \dots ,k \}$.  For a graph $G$, let $d_G(v)$ denote the degree of $v.$  For a subset $X \subseteq V(G)$, let $\partial_G(X)$ denote the set of edges with one endpoint in $X$ and the other in $V(G) \backslash X.$  We let $d_G(X) = |\partial_G(X)|$. For sets $X$ and $Y$ of vertices, we let $E_G(X,Y)$ denote the set of edges with one endpoint in $X$ and the other in $Y$ and let $e_G(X,Y) = |E_G[X,Y]|.$
  For a subset $X$ of vertices, let $N_G(X)$ denote the set of neighbours of vertices in $X$ belonging to $V(G)\backslash X.$  For a set $Y \subset V(G)$, we let $N_G(X,Y) = N_G(X) \cap Y.$
  
  For a directed graph $\overrightarrow{D},$ we let $N_{\overrightarrow{D}}^+(v)$ (resp. $N_{\overrightarrow{D}}^-(v)$) denote the set of out-neighbours (resp. in-neighbours) of $v.$

For a set $X$, we let $\mathbf{1}_X$ denote the indicator function for $X.$  For sets $X$ and $Y$, let $X \triangle Y$ denote the symmetric difference of $X$ and $Y.$

For a matroid $M$, we let $\varepsilon(M) = |E(M)|,$ the number of elements in $M$.  For convenience, if we create a new basis from a basis $B$ by deleting $X \subseteq B$ and adding the elements of a set $Y,$ then we denote the resulting basis by $B -X+Y.$  In the case where $X= \{ x_1, \dots ,x_k \}$ and $Y = \{ y_1, \dots ,y_k \},$ we will often write the new basis as $B - x_1 - \cdots - x_k + y_1 + \cdots + y_k.$
For example, if $X = \{ e \}$ and $Y = \{ f \},$ then the new basis is just $B-e+f.$

For a subset $A \subseteq E(M),$ we let $\overline{A}$ denote the set $E(M)\backslash A.$

In several places in the paper, we shall use the (easily proven) observation that if $B$ is a basis in a matroid, $C^*$ is a cocircuit, and $e\in \overline{B}$ is such that $(B+e) \cap C^* = \{ e,f \},$ then $B - f + e$ is a basis.

\section{The deficit function}
Let $M$ be a matroid of rank $n$ where $\varepsilon(M)= |E(M)| = n + k.$  
Let $\bB = ( B_1, \dots ,B_n )$ be a sequence of coloured bases of $M$.   Let $m=n+k$ and let $E(M) = \{ e_1, \dots ,e_{m} \}.$  

We shall define a bipartite graph $G = G(\bB)$ having bipartition $(U,V)$ where $U = \{ u_1, \dots ,u_n \}$, each vertex $u_i$ representing $B_i$, and $V = E(M)$.  We let $V = \{ v_1, \dots ,v_m \}$, where $v_i = e_i,\ i = 1, \dots ,m.$  For all $i,j$, the vertices $u_i$ and $v_j$ are joined by an edge in $G$ if and only if $e_j \in B_i.$
Note that $d_G(u_i) = n,\ i = 1, \dots ,n$ and $d_G(v_i) \le n,\ i = 1, \dots ,n.$
We colour the edges of $G$ with colours $1,2, \dots ,n$ such that, for $i = 1, \dots ,n$, the edges incident to $u_i$ receive colour $i.$ A rainbow basis $B$ corresponds to a {\it matching} in $G$ which saturates $U$ and those elements in $V$ corresponding to $B$. 
%

For a subset of edges $A\subseteq E(G)$, let $U_A \subseteq U$ (respectively, $V_A \subseteq V$) denote the set of endvertices  in $U$ (resp. $V$) of edges in $A.$ 

\begin{definition} For a graph $H$ of maximum degree $d$, we define the $\mathbf{H}${\bf-deficit} of a vertex $v$ to be $\delta_H(v):= d - d_H(v).$  For a subset $X\subseteq V(H)$, we define the $H$-deficit of $X$ to be
$\delta_H(X) := \sum_{x\in X}\delta_H(x).$ We define the deficit of $H$ to be $\delta(H) := \delta(V(H)).$\end{definition}

Suppose that $G$ can be partitioned into $n$ matchings $W_1, \dots ,W_n$, where, for all $i$, $V_{W_i}$ is a basis of $M$;  that is, $V_{W_i}$ is a rainbow basis.  For $i = 1, \dots ,n,$ let $G_i = G - W_1 - \dots - W_i. $   For convenience, we let $\delta_i$ denote the deficit function $\delta_{G_i}$ and we let $d_i(v) = d_{G_i}(v).$
Note that for all $j$, $d_i(u_j) = n-i.$  Also, since for each $v_j \in V,$ the edges incident to $v_j$ in $G_i$ belong to the matchings $W_{i+1}, \dots ,W_n,$ it follows that for all $j,$ $d_i(v_j) \le n-i.$
Thus $G_i$ has maximum degree $n-i.$ 

\begin{lemma}
The deficit of $G_i$ equals $k(n-i).$\label{obs1.1}
\end{lemma}

\begin{proof} We have
\begin{align*}
\delta(G_i) &= \sum_{v\in V}\delta_i(v) = \sum_{v\in V}(n-i - d_{i}(v)) = (n+k)(n-i) - \sum_{v\in V}d_{i}(v)\\
&= (n+k)(n-i) - |E(G_i)| = (n+k)(n-i) - (n^2 - ni) = k(n-i).
\end{align*}
\end{proof}

\begin{lemma}
Let $V' \subseteq V.$  Then 
$|V_{W_{i+1}} \cap V'| = |V'| -\alpha$ if and only if $\delta_{i+1}(V') = \delta_i(V') -\alpha.$\label{obs1.5}
\end{lemma}

\begin{proof} Suppose $|V_{W_{i+1}} \cap V'| = |V'| -\alpha$. Then
\begin{align*}
\delta_{i+1}(V') &= \sum_{v\in V'}((n-i-1) - d_{i+1}(v)) = (n-i-1)|V'| - \sum_{v\in V'}d_{i+1}(v)\\ &= (n-i-1)|V'| - \sum_{v\in V'}d_{i}(v) + (|V'|-\alpha)) =  (n-i)|V'| - \sum_{v\in V'}d_{i}(v) -\alpha\\
&= \delta_i(V') -\alpha.
\end{align*} The above also shows that if $\delta_{i+1}(V') = \delta_i(V') -\alpha,$ then $|V_{W_{i+1}} \cap V'| = |V'| -\alpha$.
\end{proof}

\section{The matroid $N$}

Our goal is to describe the {\it deficit inequalities} for $\delta_i$ which will be a necessary condition for the existence of $n$ disjoint rainbow bases.  Note that $e$ is a loop in $M^*$ if and only if for every base $B$ of $M$, $e\in B.$
We let $N$ be the matroid which is the simplification of the dual matroid $M^*$ plus a loop $a_{loop}.$ that is, $N = \tilde{M^*} + a_{loop}.$  
We let $[a_{loop}]$ denote the set of all loops of $M^*.$
For each element $a \in E(N)- a_{loop}$, let $[a]$ denote the $1$-flat in $M^* - [a_{loop}]$ spanned by $a$.    We define $\psi: V=E(M) \rightarrow E(N)$ where $\psi(e) = a$ if and only if $e\in [a].$ 
For any subset $F \subseteq E(N)$, let $[F] = \bigcup_{a\in F} [a].$  
  %
 We have the following observations which have straightforward proofs:
  
\begin{observation}
Let $E = E(M).$ Then we have
\begin{itemize}
\item[i)] $C$ is a circuit in $M$ if and only if $\psi(C)$ is a cocircuit in $N.$
\item[ii)] A subset $B \subseteq E$ is a basis in $M$ if and only if $\psi(E\backslash B)$ is a basis in $N.$
\item[iii)] For a subset $X \subseteq E,$ $r_N(\psi(X)) = |X| - n + r_M(E-X).$  Equivalently,\\ $r_M(X) = n - |E-X| + r_N(\psi(E-X)).$
\end{itemize}
\label{obs2.54}
\end{observation}

\subsection{Necessary conditions for the existence of $n$ rainbow bases}
 
The existence of $n$ disjoint rainbow bases necessitates that $N$ satisfy certain conditions.  These are given below in Lemma \ref{obs4}.
Suppose $G$ has a decomposition into matchings $W_1, \dots ,W_n$, where for all $i\in [n],$ $V_{W_i}$ is a basis in $M.$  Let $G_p$ be the graph obtained from $G$ by removing $W_1, \dots ,W_p.$   
For each $a \in E(N),$ we define $\delta_p(a):= \delta_p([a]).$  As before, let $E = E(M).$

For the remainder of the paper, we define $\eta := n-p.$  Note that for a matching $W$ in $G_p$, the set $\overline{V}_W$ denotes the vertices in $V$ not saturated by $W.$

\begin{lemma}
For all $e\in E,\ \delta_p(e) \ge 0,$ and for all flats $F$ in $N$ we have
$\delta_p(F) \le r_N(F)\eta.$\label{obs4}
\end{lemma}

\begin{proof}
As observed before, the maximum degree of $G_p$ equals $\eta$.  Thus $\delta_p(e) \ge 0$ for all elements $e \in E(M).$  For each matching $W_i,\ i = p+1, \dots ,n$, we have $\psi(\overline{V}_{W_i})$ is a basis of $N$ (by Observation \ref{obs2.54} ii)).  Thus for all flats $F$ in $N,$
 $|\overline{V}_{W_i} \cap [F]| \le r_N(F)$ since $\psi(\overline{V}_{W_i} \cap [F]) \subseteq F$ is independent in $N.$  It follows that
$|V_{W_i} \cap [F]| \ge |[F]| -r_N(F).$  Therefore, by Lemma \ref{obs1.5}, $\delta_{i}([F]) \ge \delta_{i-1}([F]) - r_N(F)$ and consequently,
$\delta_{i}(F) \ge \delta_{i-1}(F) -r_N(F),\ i = p+1, \dots ,n.$
Thus $0=\delta_n(F) \ge \delta_p(F) - r_N(F)\eta$ and hence $\delta_p(F) \le r_N(F)\eta.$  
\end{proof}

\section{Finding a rainbow basis}

We shall use the following well-known theorem of Rado (see \cite{Rado}):

\begin{theorem}
Suppose $X_1, \dots , X_{\ell}$ are subsets of elements of a matroid having rank function $r.$  Then there exists an independent transversal $\{ x_1, \dots ,x_\ell \}$ where
$x_i \in X_i, \ i = 1, \dots ,\ell$ if and only if for all subsets $J \subseteq [\ell],\ r\left( \bigcup_{i\in J} X_i \right) \ge |J|.$ \label{the-Rado}
\end{theorem}

In this section, we will assume that only matchings $W_1,\dots ,W_p$ have been constructed, where for $i= 1, \dots ,p$, $V_{W_i}$ is a basis of $M.$   In addition, we shall assume the matchings $W_1, \dots ,W_p$ are constructed such that $\forall e\in E(M),\ \delta_p(e) \ge 0$ and for each flat $F$ in $N$, the following {\bf deficit inequality} holds for for each flat $F$: 

\[
\delta_p(F) \le r_N(F) \eta \label{eqn-defineq} \tag{$\bigstar$}
\]
We shall prove the following key lemma:

\begin{lemma}
If $p\ge 1$, $\delta_p(e) \ge 0,\ \forall e\in E(M)$ and \eqref{eqn-defineq} holds for all flats $F$ in $N$, then there exists a matching $W_{p+1}$ in $G_p$ for which $V_{W_{p+1}}$ is a basis of $M.$
\label{lem-augment}
\end{lemma}

\begin{proof}
For $i =1, \dots ,n,$ let $X_i$ be the set of neighbours of $u_i$ in $V.$  If one can find and independent transversal $\{ x_1, \dots ,x_n \}$ where
$x_i \in X_i, \ i = 1, \dots ,n$, then the corresponding matching $W_{p+1}$ will be such that $V_{W_{p+1}}$ is a basis.  To show such a transversal exists, it suffices to show (by Rado's theorem) that for all $J \subseteq [n],$ $r_M\left( \bigcup_{i\in J} X_i \right) \ge |J|.$  Let $J \subseteq [n]$ and let $X = \bigcup_{i\in J} X_i.$  
By Observation \ref{obs2.54} iii), we have $r_M(X) = n - |E-X| + r_N(\psi(E-X)).$  We also have by (\ref{eqn-defineq}) that 
$$r_N(\psi(E-X)) \ge \frac {\delta_p(E-X)}\eta = \frac {\sum_{v\in E-X} (\eta - d_p(v))}\eta = |E-X| - \frac {\sum_{v\in E-X} d_p(v)}\eta.$$
Thus $r_M(X) \ge n - \frac {\sum_{v\in E-X} d_p(v)}\eta.$  Let $U_J = \{ u_i \in U \ \big| \ i \in J \}$.  Given that each vertex $v \in E-X$ is only adjacent to vertices in $U - U_J,$ it follows that $\sum_{v\in E-X} d_p(v) \le \sum_{u_i \in U_J}d_p(u_i) \le \eta (n-|J|).$
Thus it follows from the above that $r_M(X) \ge n - \frac {\eta (n- |J|)}\eta = |J|.$  This completes the proof. 
\end{proof}

 \section{Changing matchings using alternating paths}
 
 We shall assume $\forall e\in E,$ $\delta_p(e) \ge 0$ and for all flats $F$ in $N$, (\ref{eqn-defineq}) holds.
 By Lemma \ref{lem-augment}, one can find a matching $W= W_{p+1}$ in $G_p$ for which $V_{W}$ is a basis in $M.$  Our strategy is show that, provided $\eta = n-p \ge k^3,$ one can change $W$ using {\it alternating paths} in such a way that when we delete it from $G_p$, the resulting graph $G_{p+1}$ will also be such that $\forall e\in E(M), \ \delta_{p+1}(e) \ge 0$ and (\ref{eqn-defineq}) holds with $\eta -1$ in place of $\eta.$
It will then follow by Lemma \ref{lem-augment} that $G_{p+1}$ has a matching $W_{p+2}$ for which $V_{W_{p+2}}$ is a basis.  The process then continues with $W_{p+2}$ in place of $W_{p+1}.$  

Let $A = \psi(\overline{V}_W)$, which (by Observation \ref{obs2.54} ii)) is a basis in $N.$  It follows by Lemma \ref{obs1.5} that for all flats $F$ in $N,$
$$\delta_{p+1}(F) = \delta_p(F) - \nu \  \Longleftrightarrow \ |\overline{V}_{W} \cap [F]| = \nu.$$  In particular, to ensure that for all flats $F$ in $N$, $\delta_{p+1}(F) \le r_N(F)(\eta-1),$ it suffices to choose $W$ so that $|\overline{V}_{W} \cap [F]| \ge \delta_p(F) - r_N(F)(\eta-1).$  Noting that $|\overline{V}_{W} \cap [F]| = |A \cap F|$, this condition is 
equivalent to ensuring that for all flats $F$, 
\[
|A \cap F| \ge \delta_p(F) - r_N(F)(\eta-1).\label{eqn-Fineq} \tag{$\bigstar\bigstar$} 
\]
 %

Given a set of edges $X$ forming a matching in a bipartite graph, an {\bf alternating path} with respect to $X$ is a path $P$ whose edges alternate between matching edges and non-matching edges.  

Let $W$ be a matching in $G_p$ where $V_{W}$ is a basis.  For a subset $Y \subseteq V_W,$  we let $U_{W,Y}$ denote the set of vertices in $U$ which are reachable by an alternating path with respect to $W$ originating at some vertex in $Y$ where the first edge of the path is in $W.$  We are interested in $e_p(U_{W,Y},\overline{V}_W)$, the number of edges between $U_{W,Y}$ and $\overline{V}_W.$

In Lemmas \ref{lem3} and \ref{obs6} below we show that if a flat $F$ fails to satisfy (\ref{eqn-Fineq}); that is, it's deficit is too large, then there are alternating paths from vertices in $\overline{V}_W$ to vertices in $[F]$. We shall exploit the existence of such  alternating paths to show that one can alter $W$ so that $W$ becomes a matching $W'$ where $A' = \psi(\overline{V}_{W'})$ and $|A' \cap F| > |A \cap F|,$ thereby reducing the number of flats which fail to satisfy (\ref{eqn-Fineq}).

\begin{lemma}
 $e_p(U_{W,Y},\overline{V}_W) \ge \delta_p(Y).$
\label{obs-alt}
\end{lemma}

\begin{proof}
Let $Y'$ be the set of vertices in $V_W$ matched to vertices in $U_{W,Y}.$  Note that $|Y'| = |U_{W,Y}|.$  Since each vertex in $U_{W,Y}$ has degree $\eta,$ we have 
\begin{align*}
\eta|U_{W,Y}| &= d_p(U_{W,Y}) = e_p(U_{W,Y}, Y' \cup \overline{V}_W) = e_p(U_{W,Y},Y') + e_p(U_{W,Y},\overline{V}_W)\\
&\le d_p(Y') + e_p(U_{W,Y},\overline{V}_W)\\
&=|Y'|\eta - \delta_p(Y') + e_p(U_{W,Y},\overline{V}_W)\\
&=  |U_{W,Y}|\eta - \delta_p(Y') + e_p(U_{W,Y},\overline{V}_W)
\end{align*}
From the above, we see that $e_p(U_{W,Y},\overline{V}_W) \ge \delta_p(Y') \ge \delta_p(Y).$
\end{proof}

\begin{definition} For a subset $Y\subseteq V,$ let $\tilde{Y} = \{ v \in Y \cap V_W\ \big| \ \delta_p(v) >0\}.$   An alternating path with respect to $W$ in $G_p$ which originates in $\overline{V}_W$ and terminates in $\tilde{Y}$  is called a $\mathbf{Y}${\bf -path} with respect to $W$.\end{definition}

\begin{lemma}
Let $Y \subseteq V.$  There are at least $\big\lceil \frac { \delta_p(\tilde{Y})}{\eta} \big\rceil$ vertices $v\in \overline{V}_W$ for which there is a $Y$-path originating at $v$ and terminating at some vertex in $\tilde{Y}$.
\label{lem3}
\end{lemma}

\begin{proof}
By Lemma \ref{obs-alt}, $e_{p}(U_{W,\tilde{Y}},\overline{V}_W) \ge \delta_p(\tilde{Y}).$  Given that each vertex in $\overline{V}_W$ has degree at most $\eta$,  there are at least 
$\big\lceil \frac {e_{p}(U_{W,\tilde{Y}},\overline{V}_W)}{\eta} \big\rceil \ge \big\lceil \frac {\delta_p(\tilde{Y})}{\eta} \big\rceil$ vertices in $\overline{V}_W$ which are neighbours of vertices in $U_{W,\tilde{Y}}.$   The lemma now follows since if $v\in \overline{V}_W$ is a neighbour of $U_{W,\tilde{Y}'}$, then there is also a $Y$-path originating at $v$ and terminating at some vertex of $\tilde{Y}.$  \end{proof}

\begin{lemma}
Let $Y\subseteq V$ and suppose $Z \subset \overline{V}_W$ is such that $\delta_p(\tilde{Y}) > i\eta + d_p(Z)$ for some nonnegative integer $i.$
Then there are at least $i +1$ vertices $v$ in $\overline{V}_W - Z$ for which some $Y$-path originates at $v$ and terminates at some vertex in $\tilde{Y}$. \label{obs3}
\end{lemma}

\begin{proof}
 It follows by Lemma \ref{obs-alt} that $e_{p}(U_{W,\tilde{Y}},\overline{V}_W) \ge \delta_p(\tilde{Y})  > i\eta + d_p(Z)$.  Thus it follows that there are at least 
 $\left\lceil \frac {e_{p}(U_{W,\tilde{Y}},\overline{V}_W) - d_p(Z)}{\eta} \right\rceil \ge \left\lceil \frac {i\eta+1}{\eta}\right\rceil = i+1$ vertices in $N_p(U_{W,\tilde{Y}}, \overline{V}_W -Z).$ 
This in turn implies that there are at least $i+1$ vertices $v$ in $\overline{V}_W - Z$ for which there is a $Y$-path originating at $v$ and terminating at some vertex $\tilde{Y}.$
\end{proof}
   
 %
 \begin{lemma}
 Let $Y \subseteq V$ and suppose that $|Y \cap \overline{V}_W| =s.$  If $\delta_p(Y) \ge \ell + s\eta,$ then 
 there exist at least $\lceil \frac {\ell}{\eta} \rceil$ vertices $v\in \overline{V}_W - Y$ for which there is a $Y$-path originating at $v.$ \label{obs6}
 \end{lemma}
 
 \begin{proof}
 Since $\delta_p(Y) \ge \ell + s\eta,$ we have 
 \begin{align*}
 \delta_p(Y) &= \delta_p(\tilde{Y}) + \delta_p(Y\cap \overline{V}_W) \ge \ell + s\eta\\
 \delta_p(\tilde{Y}) &\ge \ell + s\eta - \delta_p(Y \cap \overline{V}_W)\\
 &= \ell + s\eta - \sum_{v_i \in Y\cap \overline{V}_W}(\eta - d_p(v_i))\\
 &= \ell + d_p\left( Y\cap \overline{V}_W\right).
 \end{align*}
 It now follows by Lemma \ref{obs3} that there are at least $\lceil \frac {\ell}{\eta} \rceil$ vertices $v\in \overline{V}_W - Y$ for which there is a $Y$-path originating at $v.$
 \end{proof}
 
 If there is an alternating path $P$ with respect to $W$ which originates a vertex $z_1\in \overline{V}_W$ and terminates at a vertex $z_1' \in V_W$, then we write $z_1 \xrightarrow [W]{P} z_1'.$  We will simply write  $z_1 \xrightarrow [W]{} z_1'$,  if $z_1 \xrightarrow [W]{P} z_1',$ for some alternating path $P.$ 
  The next observation is an easily proven property of alternating paths.
 
 \begin{observation}
 Suppose $z_1 \xrightarrow [W]{P_1} z_1'$ and $z_2 \xrightarrow [W]{P_2} z_2'.$  
 If $V(P_1) \cap V(P_2) \ne \emptyset,$ then $z_1 \xrightarrow [W]{} z_2'$ and $z_2 \xrightarrow [W]{} z_1'.$\label{obs-alternating1}
 \end{observation}
 
 \section{An example}
 
 To illustrate how one can use Lemmas \ref{lem3} and \ref{obs6} to change the matching $W$ so that the flats in $N$ satisfy (\ref{eqn-Fineq}).   Suppose 
 \begin{itemize}
 \item  $N \simeq F_7^-$ ($F_7$ minus an element)
 \item $E(N) = \{ a,b,c,d,e,f \}$ 
 \item $A = \psi(\overline{V}_W) = \{ a,b,c \}$ and $\eta = 3.$
 \item $\delta_p(b) = \delta_p(c) = \delta_p(e) =1$ and $\delta_p(a) = \delta_p(d) = \delta_p(f) =2.$
 \item $\overline{V}_W = \{ y_1, y_2, y_3 \}$, $a = \psi(y_1), \ b = \psi(y_2),$ and $c = \psi(y_3).$
 \end{itemize}
 See figure \ref{example}. 
 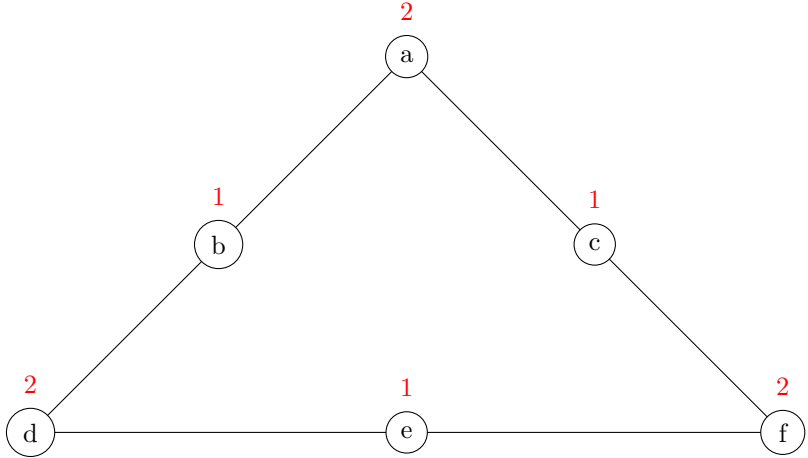
\begin{figure}
\centering
\begin{tikzpicture}[scale=0.50, every node/.style={draw=black, circle,fill=white}]
\draw (10,10) -- (0,0); \draw (0,0) -- (20,0) -- (10,10);\node[label={\textcolor{red}{2}}] (a) at (10,10) {a};
\node[label={\textcolor{red}{1}}] (b) at (5,5) {b};
\node[label={\textcolor{red}{1}}] (c) at (15,5) {c};
\node[label={\textcolor{red}{2}}] (d) at (0,0) {d};
\node[label={\textcolor{red}{1}}] (e) at (10,0) {e};
\node[label={\textcolor{red}{2}}] (f) at (20,0) {f};
\end{tikzpicture}
\caption{Matroid $N$ with deficits for each element}\label{example}
 \end{figure}
We observe that (\ref{eqn-Fineq}) holds for all flats of $N$ except $F_0 = \{ d,e,f \}.$  We describe how one can change $W$ using alternating paths so that (\ref{eqn-Fineq}) holds for $F_0$ as well.    By Lemma \ref{lem3}, for there is a $[d]$-path $P_1$ from some vertex 
 $y_i \in \overline{V}_W$ to a vertex $y \in [d].$  Suppose $i=1;$ that is, $y_1 \xrightarrow [W]{P_1} y.$  Let $W'$ be the matching which is the symmetric difference $W' = W \triangle E(P_1).$  Then $A' = \psi(\overline{V}_{W'}) = A - a + d = \{ b,c,d \}$, which is a basis for $N.$  Furthermore, (\ref{eqn-Fineq}) is seen to hold for all flats in $N$ when $A$ is replace by $A'$.  If $i =2,$ then a similar matching $W'$ can be found.  Thus we may assume $i=3$; that is, $y_3 \xrightarrow [W]{P_1} y.$  In this case, we see that for the matching $W' = W \triangle E(P_1),$  the set $A' = \psi(\overline{V}_{W'}) = A - c + d = \{ a,b,d \}$ is not a basis.  However, we can repeat the same arguments with $f$ in place of $d$.  That is, we may assume that there is a $[f]$-path $P_2$ where $y_1 \xrightarrow [W]{P_1} y'.$  Suppose that $P_i,\ i = 1,2$ are vertex-disjoint.  Let $W'$ be the matching $W' \triangle E(P_1 \cup P_2).$  Then $A' = \psi(\overline{V}_{W'}) = A - b - c + d + f = \{ a, d, f \}$ is a basis for $N$ and furthermore, (\ref{eqn-Fineq}) is seen to hold for all flats in $N$ when $A$ is replaced by $A'$.  Suppose $P_1$ and $P_2$ have common vertices.  Then it follows by Observation \ref{obs-alternating1} that $y_3 \xrightarrow[W]{P} y'$ for some alternating path $P.$  Letting $W'$ be the matching
 $W' = W  \triangle E(P),$ we see that $A' = \psi(\overline{V}_{W'}) = A - b + d = \{ a,c,d \}$ is a basis for $N$ and furthermore, (\ref{eqn-Fineq}) is seen to hold for all flats in $N$ when $A$ is replaced by $A'$
 
  \section{The set of flats $\frakF$}\label{sec-Fset}
  
  Let $W$ be a matching in $G_p$ for which $E_W$ is a basis and let $\overline{V}_{W} = \{ y_1, \dots ,y_k \}$ .  For all $i \in [k],$ let $a_i = \psi(y_i)$ and let $A = \{ a_1, \dots , a_k \},$ which is a basis in $N$.   We shall assume that for all flats $F$ in $N$, $\delta_p(F) \le  r_N(F) \eta.$  Our primary goal is to show that $W$ can be chosen so that for all flats $F$ in $N$, (\ref{eqn-Fineq}) holds.
When $\delta_p(F) \le r_N(F)(\eta -1),$ this inequality automatically holds and thus the focus will be on the flats $F$ where
$\delta_p(F) > r_N(F)(\eta -1).$ 

\begin{definition}Let $\F$ be the set all flats $F$ in $N$ for which $\delta_p(F) > r_N(F)(\eta-1)$ and let $\frakF$ denote the set of all flats in $N$ which can be expressed as a finite intersection of flats in $\F.$ 
\end{definition}
We note that $N\in \F$ and $\F \subseteq \frakF$.   Furthermore, $\frakF$ is closed under intersections. 

Suppose $F\in \F$ and let $\delta_p(F) = r_N(F)\eta - i$ (where $0 \le i < r_N(F)$).  If $|A \cap F| < r_N(F)-i,$ then inequality (\ref{eqn-Fineq}) fails to hold for $F.$  As such we wish to alter $W$ so that for the corresponding 
set $A$, $|A \cap F| \ge r_N(F)-i.$  The process of changing $W$ is complicated by the fact that in doing so, the inequality (\ref{eqn-Fineq}) may not be preserved for other flats in $\F$ for which it holds.  To get around this problem, we will simply alter $W$ by successive iterations so as to obtain a matching $W$ such that for all flats $F\in \frakF,$ $|A \cap F| = r_N(F).$  We will show that it is possible to choose $W$ as such provided $\eta \ge k^3.$ 

\begin{definition}
For all $i \in [k]$ let $H_i$ be the hyperplane in $N$, where $H_i = \mathrm{cl}_N (A - a_i)$, and let $C_i^* = \overline{H}_i,$ which is a cocircuit in $N$ containing $a_i.$  
\end{definition}

\begin{observation}
For $\{ i_1, i_2, \dots ,i_s\} \subseteq [k]$,  $\delta_p \left(\bigcup_{j=1}^s C_{i_j}^* \right) \ge s\eta.$\label{obs-deltabound}
\end{observation}

\begin{proof}
We have $N_{i_1 \cdots i_s} = H_{i_1} \cap H_{i_2} \cap \cdots \cap H_{i_s} = \mathrm{cl}_N (A - \{ a_{i_1}, \dots , a_{i_s} \} )$ is a rank-$(k-s)$ flat in $N.$ By assumption, $\delta_p(N_{i_1 \cdots i_s}) \le (k-s)\eta.$  Thus $\delta_p \left(\bigcup_{j=1}^s C_{i_j}^* \right) =  \delta_p(\overline{N_{i_1 \cdots i_s}}) = \delta_p(N) - \delta_p(N_{i_1 \cdots i_s}) \ge s\eta.$ 
\end{proof}
  
The next definition exploits the useful property that $\frakF$ is closed under taking intersections. 
 
 \begin{definition} For all $i \in [k]$, there is a unique flat in $\frakF$ of least rank which contains $a_i.$  We shall denote such a flat by $F_i.$  Furthermore, we let $F_i' = C_i^* \cap F_i.$   
 \end{definition}

\begin{definition} For a subset $F \subseteq E(N)$, we define $\omega_{W}(F) := |A \cap F|.$\end{definition}   
Our ultimate goal is to show that, provided $\eta \ge k^3,$  one can alter the matching $W$, obtaining a matching $W'$ such that $V_{W'}$ is a basis and for all flats $F\in \frakF$ in $N,$ $\omega_{W'}(F) = r_N(F).$  This in turn will imply that (\ref{eqn-Fineq}) holds for all flats $F$ in $N.$
 
\begin{definition} For all flats $F$ in $N$, we define $F_A := \mathrm{cl}_N(F\cap A).$\end{definition}  
 We note that $r_N(F_A) = \omega_W(F).$
 \begin{lemma}
Suppose that $Q_i,\ i \in [s],$ are flats in $N$ where for all $i\in [s],$ $\delta_p(Q_i) = r_N(Q_i)\eta - j_i.$
Then $\delta_p\left(\bigcap_{i=1}^s Q_{i} \right) \ge r_N\left( \bigcap_{i=1}^s Q_{i}\right) \eta - j_1 - j_2 - \cdots - j_s.$ 
\label{lem-B1}
\end{lemma}

\begin{proof}
Using the fact that $r_N$ is a submodular function, we have
\begin{align*}
\delta_p(Q_1 \cap Q_2) &= \delta_p(Q_1) + \delta_p(Q_2) - \delta_p(Q_1 \cup Q_2)\\
&= r_N(Q_1)\eta - j_1 + r_N(Q_2)\eta - j_2 -  \delta_p(Q_1 \cup Q_2)\\
&\ge (r_N(Q_1) + r_N(Q_2))\eta - j_1 -j_2 - r_N(Q_1 \cup Q_2)\eta\\
&= (r_N(Q_1) + r_N(Q_2) - r_N(Q_1 \cup Q_2))\eta - j_1 - j_2\\
&\ge r_N(Q_1 \cap Q_2)\eta - j_1 - j_2.
\end{align*} 
It is now easily proven by induction that $\delta_p(Q_1 \cap Q_2 \cap \cdots \cap Q_s) \ge r_N(Q_1 \cap Q_2 \cap \cdots \cap Q_s)\eta - j_1 -j_2 - \cdots - j_s.$
\end{proof} 

\begin{lemma}
For all $F\in \frakF,$ $\delta_p(F) \ge r_N(F)\eta -(k-1)^2$.\label{obs-lowerbound}
\end{lemma}

\begin{proof}
By the definition of $\F$, we have for all $F\in \F$, $\delta_p(F) > r_N(F)(\eta -1) \ge r_N(F)\eta -k$.
Let $F \in \frakF.$ Then $F = F_1 \cap F_2 \cap \cdots \cap F_\ell$ for some flats $F_i \in \F,\ i =1,2, \dots ,\ell.$  Assuming that $\ell$ is a minimum number of such flats, it follows that for $i = 1, \dots ,\ell-1$,
$r_N \left( \bigcap_{j\le i} F_j \right) > r_N \left( \bigcap_{j\le i+1} F_j \right).$  Given that $r_N(F) \ge 1,$ it follows that $\ell \le k-1.$  It now follows by Lemma \ref{lem-B1} that $\delta_p(F) \ge r_N(F)\eta - (k-1)^2.$
\end{proof}

\begin{lemma}
For all $i \in [k]$, $\delta_p(F_i') \ge \eta - (k-1)^2.$
\label{lem-B2} 
\end{lemma}

\begin{proof}
We have that 
\begin{align*}
\delta_p(F_i') &= \delta_p(F_i) - \delta_p(F_i \cap H_i) \ge r_N(F_i)\eta - (k-1)^2 - r_N(F_i\cap H_i)\eta\\
&\ge  r_N(F_i)\eta - (k-1)^2 - (r_N(F_i)-1)\eta\\
&= \eta - (k-1)^2
\end{align*}
\end{proof}
  
\begin{lemma}
Let $A' \subseteq A$ where for all $a_i, a_j \in A'$, $F_i = F_j.$  Then $\delta_p\left(\bigcup_{a_i \in A'}F_i'\right) \ge |A'|\eta - (k-1)^2.$
\label{lem-B3}
\end{lemma}

\begin{proof}
Let $F \in \frakF$ be the flat for which $F = F_i$, for all $a_i \in A'.$  Let $F'' = F \cap \bigcap_{a_i \in A'}H_i.$  Then we have
\begin{align*}
\delta_p\left(\bigcup_{a_i \in A'}F_i'\right) &= \delta_p(F) - \delta_p(F'')\\
&\ge r_N(F)\eta - (k-1)^2 - r_N(F'')\eta\\
&\ge r_N(F)\eta - (k-1)^2 - (r_N(F) - |A'|)\eta\\
&= |A'|\eta - (k-1)^2.
\end{align*}
\end{proof}

\begin{lemma}
 Let $A' \subseteq A,$ where for all $a_i \in A',$ $\omega_W(F_i) = r_N(F_i).$  Then $\delta_p\left(\bigcup_{a_i \in A'}F_i'\right) \ge |A'|\eta - k(k-1)^2.$
\label{lem-B4}
\end{lemma}

\begin{proof}
We first observe that for all $a_i, a_j \in A',$ if $F_i' \cap F_j' \ne \emptyset,$ then $a_i \in F_j,$ $a_j\in F_i$ and hence $F_i = F_j.$  Defining an equivalence relation on $A'$ where $a_i \sim a_j$ if and only if $F_i = F_j$,  let $A_1', \dots ,A_s'$
be the equivalence classes.  Then we have by Lemma \ref{lem-B2},
\begin{align*}
\delta_p\left(\bigcup_{a_i \in A'}F_i'\right) &= \sum_{i=1}^s \delta_p\left(\bigcup_{a_j \in A_j'}F_i'\right)\\
&\ge \sum_{i=1}^s(|A_i'|\eta - (k-1)^2)\\
&= |A'|\eta - s(k-1)^2 \ge |A'|\eta - k(k-1)^2.
\end{align*}
\end{proof}

\section{Path-tangles and path-chains}

To alter $W$ so that 
for all $F\in \frakF,$ $\omega_{W}(F) = r_N(F)$,  we will do so incrementally using sequences of alternating paths in $G_p$ which we refer to as {\it path-chains}.  

Recall that $\overline{V}_W = \{ y_1, \dots ,y_k \}$, $\psi(\overline{V}_W) = A = \{ a_1, \dots ,a_k \}$ (a basis for $N$) and for $i\in [k],$ $C_i^* = \overline{H}_i$, $F_i' = C_i^* \cap F_i$. 
Let $A_0 \subseteq E(N)-A$ and $Y_0 = [A_0].$

Suppose $P_1, \dots , P_t$ are vertex-disjoint alternating paths with respect to $W$ where for $i =1, \dots ,t$, $y_i \xrightarrow[W]{P_i} y_i'$ (and $y_i' \in V_W$).  Let $a_i' = \psi(y_i'), \ i = 1, \dots ,t.$  We have that $W' = W \triangle E(P_1 \cup \cdots \cup P_t)$ is a matching (where $|W'| = |W|$)
and $\overline{V}_{W'} = \overline{V}_W - \{ y_1, \dots ,y_t \} + \{ y_1', \dots ,y_t' \}$.  Suppose $\overline{V}_{W'}$ is a basis.  Then it follows that $A' = A - \{ a_1, \dots ,a_t \}  + \{ a_1', \dots ,a_t' \}$ is a basis of 
$N$.  Our goal is to show that such paths $P_1, \dots ,P_t$ can be found so that $\omega_{W'}(F) > \omega_W(F)$ (i.e. $|F \cap A'| > |F \cap A|$) for certain flats $F\in \frakF.$   Ultimately, this procedure  will result in (\ref{eqn-Fineq}) being satisfied for all flats $F\in \frakF.$  To obtain $A'$ from $A$, for $i = 1, \dots ,t,$ we replace $a_i$ with $a_i'$.  Generally, to ensure that this procedure yields a basis $A'$, we need to choose the path $P_1, \dots ,P_t$ strategically.   
In particular, for $i =1, \dots ,t,$ we will choose $P_i$ such that $\psi(y_i') = a_i' \in C_i^*.$

\begin{definition}
For a positive integer $t,$ let $\beta: [t] \rightarrow [k]$ be a function and suppose $P_1, \dots ,P_t$ are alternating paths with respect to $W$ originating in $\overline{V}_W$ and terminating in $V_W.$ The $t$-tuple $K = (P_1, \dots ,P_t)$ is called a {\bf path-tangle} rooted at $Y_0$ if:
\begin{itemize}
\item[i)]  For all $i\in [t],$ $y_{\beta(i)} \xrightarrow [W]{P_i} y_{\beta(i)}'$.
\item[ii)] $y_{\beta(1)}' \in Y_0$.
\item[iii)] For $i=2, \dots ,t,$  either $y_{\beta(i)}' \in Y_0$ or $y_{\beta(i)}' \in [C_{\beta(i')}^*]$, for some $i' < i.$
\end{itemize}
\end{definition}

\begin{definition} $K=(P_1, P_2, \dots ,P_t)$ is called a {\bf  path-chain} if instead of iii) $K$ satisfies the stronger property:
\begin{itemize}
\item[iv)] For $i=2, \dots ,t,$  $y_{\beta(i)}' \in [C_{\beta(i-1)}^*]$.
\end{itemize}
\end{definition}

Let $K = (P_1, \dots ,P_t)$ be the path-tangle as described above.  Consider $y_{\beta(i)}.$  Within $K$ one can find a path-chain $K' = (P_{\gamma(1)}, \dots ,P_{\gamma(s)})$ where $\gamma(s) = i$ and 
$y_{\gamma(1)}' \in Y_0.$
One can describe such a path-chain informally as follows:  The path $P_{\beta(i)}$ connects $y_{\beta(i)}$ with $y_{\beta(i)}'$.  If $y_{\beta(i)}' \in Y_0,$ then
$K' = P_{\beta(i)}$ is the desired path-chain.  Suppose $y_{\beta(i)}' \not\in Y_0.$  Then $y_{\beta(i)}' \in [C_{\beta(i')}^*]$, for some $i' < i.$  The path $P_{\beta(i')}$ connects $y_{\beta(i')}$ with $y_{\beta(i')}'.$  If $y_{\beta(i')}' \in Y_0$, then $K' = (P_{\beta(i')}, P_{\beta(i)})$ is the desired path-chain.  Suppose $y_{\beta(i')}' \not\in Y_0$.  Then $y_{\beta(i')}'  \in [C_{\beta(i'')}^*]$, for some $i'' < i'.$  The path $P_{\beta(i'')}$ connects $y_{\beta(i'')}$ with $y_{\beta(i'')}'$.  If $y_{\beta(i'')}' \in Y_0$, then $K' = (P_{\beta(i'')}, P_{\beta(i)'}, P_{\beta(i)})$ is the desired chain.  If $y_{\beta(i'')}' \not\in Y_0,$ then the process continues.  Seeing as these paths must eventually terminate in $Y_0,$ we will eventually attain the desired path-chain $K'$.
We refer to $K'$ as a {\bf path-chain in} $\mathbf{K}$ {\bf from} $\mathbf{y_{\beta(i)}}$ {\bf to} $\mathbf{Y_0}$ . 

\subsection{The composition of path-chains}

For positive integers $t_i,$ $i=1,2,$ let $\beta_i: [t_i] \rightarrow [k]$ be functions and for $i=1,2,$ let $K_i = (P_{i1}, \dots ,P_{it_i})$ be a path-chain rooted at $Y_0,$ where for $j = 1, \dots ,t_i$, $y_{\beta_i(j)} \xrightarrow [W] {P_{ij}} y_{\beta_i(j)}^i.$
Suppose that $y_{\beta_1(1)}^1 \in [C_{\beta_2(t_2)}^*]$. Then $K = (P_{21}, \dots ,P_{2t_2}, P_{11}, \dots ,P_{1t_1})$ is seen to be a path-chain rooted at $Y_0.$ 
We call such a path-chain the {\bf composition} of $K_2$ with $K_1$ and denote it by $K_1 \circ K_2.$  See Figure \ref{fig1}.

\begin{figure}
\centering
\begin{tikzpicture}[scale=1.25]
\draw[fill=green] (0,0) ellipse (2 and .25) node[left=2.5cm]{$[C_{\beta_1(t_1)}^*]$};
\node at (-1.3,0){$y_{\beta_1(t_1)}$};
\draw[fill=black] (-.85,0) circle(0.04);

\draw[fill=green] (0,-1) ellipse (2 and .25) node[left=2.5cm]{$[C_{\beta_1(t_1-1)}^*]$};
\node at (-1.4,-1){$y_{\beta_1(t_1-1)}$}; \node at (1.4,-1){$y_{\beta_1(t_1)}^1$};

\draw[fill=black] (-.85,-1) circle(0.04);
\draw[fill=black] (.85,-1) circle(0.04);

\draw (-.85,0)  .. controls (-.70,-.60) and (-.15,-.60) ..  (0,-.5)  .. controls (.15, -.30) and (.70, -.30) ..  node[below]{\small{$P_{1t_1}$}}(.85,-1);

\draw[fill=green] (0,-2) ellipse (2 and .25) node[left=2.5cm]{$[C_{\beta_1(t_1-2)}^*]$};
\node at (-1.40,-2){\small{$y_{\beta_1(t_1-2)}$}}; \node at (1.40,-2){\small{$y_{\beta_1(t_1-1)}^1$}};
\draw (-.85,-1)  .. controls (-.70,-1.60) and (-.15,-1.60) ..  (0,-1.5)  .. controls (.15, -1.30) and (.70, -1.30) ..  node[below]{\small{$P_{1(t_1-1)}$}}(.85,-2);
\draw[fill=black] (-.85,-2) circle(0.04);
\draw[fill=black] (.85,-2) circle(0.04);

\draw[fill=green] (0,-3) ellipse (2 and .25) node[left=2.5cm]{$[C_{\beta_1(t_1-3)}^*]$};
\node at (-1.40,-3){\small{$y_{\beta_1(t_1-3)}$}}; \node at (1.40,-3){\small{$y_{\beta_1(t_1-2)}^1$}};
\draw (-.85,-2)  .. controls (-.70,-2.60) and (-.15,-2.60) ..  (0,-2.5)  .. controls (.15, -2.30) and (.70, -2.30) ..  node[below]{\small{$P_{1(t_1-2)}$}}(.85,-3);
\draw[fill=black] (-.85,-3) circle(0.04);
\draw[fill=black] (.85,-3) circle(0.04);

\draw[fill=green] (0,-5) ellipse (2 and .25) node[left=2.5cm]{$[C_{\beta_1(1)}^*]$};
\node at (-1.40,-5){\small{$y_{\beta_1(1)}$}}; \node at (1.40,-5){\small{$y_{\beta_1(2)}^1$}};
\draw (-.85,-4)  .. controls (-.70,-4.60) and (-.15,-4.60) ..  (0,-4.5)  .. controls (.15, -4.30) and (.70, -4.30) ..  node[below]{\small{$P_{12}$}}(.85,-5);
\draw[fill=black] (-.85,-5) circle(0.04);
\draw[fill=black] (.85,-5) circle(0.04);

\node at (0,-3.5){\Large{$\cdot$}};
\node at (0,-3.75){\Large{$\cdot$}};
\node at (0,-4){\Large{$\cdot$}};

\draw[fill=green] (5,0) ellipse (2 and .25) node[right=2.5cm]{$[C_{\beta_2(t_2)}^*]$};
\node at (3.60,0){\small{$y_{\beta_2(t_2)}$}}; \node at (6.40,0){\small{$y_{\beta_1(1)}^1$}};
\draw[fill=black] (4.15,0) circle(0.04);
\draw[fill=black] (5.85,0) circle(0.04);

\draw (-.85,-5)  .. controls (.8,-5.5) and (2.3,-5.5) ..  (2.5, -5) .. controls (3, -4) and (3, -1) ..  node[left]{\small{$P_{11}$}}(2.5,0);
\draw (2.5,0) .. controls (2.2,.5) and (2.35,.75) .. (2.6,1) .. controls (3.4,1.55) and (4.85, .75) .. (5.85,0);

\draw[fill=green] (5,-1) ellipse (2 and .25) node[right=2.5cm]{$[C_{\beta_2(t_2-1)}^*]$};
\node at (3.60,-1){\small{$y_{\beta_2(t_2-1)}$}}; \node at (6.40,-1){\small{$y_{\beta_2(t_2)}^2$}};
\draw (4.15,0)  .. controls (4.30,-.60) and (4.85,-.60) .. (5,-.5)  .. controls (5.15, -.30) and (5.70, -.30) ..  node[below]{\small{$P_{2t_2}$}}(5.85,-1);
\draw[fill=black] (4.15,-1) circle(0.04);
\draw[fill=black] (5.85,-1) circle(0.04);

\draw[fill=green] (5,-2) ellipse (2 and .25) node[right=2.5cm]{$[C_{\beta_2(t_2-2)}^*]$};
\node at (3.60,-2){\small{$y_{\beta_2(t_2-2)}$}}; \node at (6.40,-2){\small{$y_{\beta_2(t_2-1)}^2$}};
\draw (4.15,-1)  .. controls (4.30,-1.60) and (4.85,-1.60) .. (5,-1.5)  .. controls (5.15, -1.30) and (5.70, -1.30) ..  node[below]{\small{$P_{2(t_2-1)}$}}(5.85,-2);
\draw[fill=black] (4.15,-2) circle(0.04);
\draw[fill=black] (5.85,-2) circle(0.04);

\draw[fill=green] (5,-3) ellipse (2 and .25) node[right=2.5cm]{$[C_{\beta_2(t_2-3)}^*]$};
\node at (3.60,-3){\small{$y_{\beta_2(t_2-3)}$}}; \node at (6.40,-3){\small{$y_{\beta_2(t_2-2)}^2$}};
\draw (4.15,-2)  .. controls (4.30,-2.60) and (4.85,-2.60) .. (5,-2.5)  .. controls (5.15, -2.30) and (5.70, -2.30) ..  node[below]{\small{$P_{2(t_2-2)}$}}(5.85,-3);
\draw[fill=black] (4.15,-3) circle(0.04);
\draw[fill=black] (5.85,-3) circle(0.04);

\draw[fill=green] (5,-5) ellipse (2 and .25) node[right=2.5cm]{$[C_{\beta_2(1)}^*]$};
\node at (3.60,-5){\small{$y_{\beta_2(1)}$}}; \node at (6.40,-5){\small{$y_{\beta_2(2)}^2$}};
\draw (4.15,-4)  .. controls (4.30,-4.60) and (4.85,-4.60) .. (5,-4.5)  .. controls (5.15, -4.30) and (5.70, -4.30) ..  node[below]{\small{$P_{22}$}}(5.85,-5);
\draw[fill=black] (4.15,-5) circle(0.04);
\draw[fill=black] (5.85,-5) circle(0.04);

\draw[fill=green] (5,-6) ellipse (2 and .25) node[right=2.5cm]{$Y_0$};
\node at (6.40,-6){\small{$y_{\beta_2(1)}^2$}};
\draw (4.15,-5)  .. controls (4.30,-5.60) and (4.85,-5.60) .. (5,-5.5)  .. controls (5.15, -5.30) and (5.70, -5.30) ..  node[below]{\small{$P_{21}$}}(5.85,-6);
\draw[fill=black] (5.85,-6) circle(0.04);

\node at (5,-3.5){\Large{$\cdot$}};
\node at (5,-3.75){\Large{$\cdot$}};
\node at (5,-4){\Large{$\cdot$}};
\end{tikzpicture}
\caption{The composition of the path-chain $K_2 = (P_{21}, \dots ,P_{2t_2})$ with $K_1 = (P_{11}, \dots ,P_{1t_1})$}\label{fig1}
\end{figure}
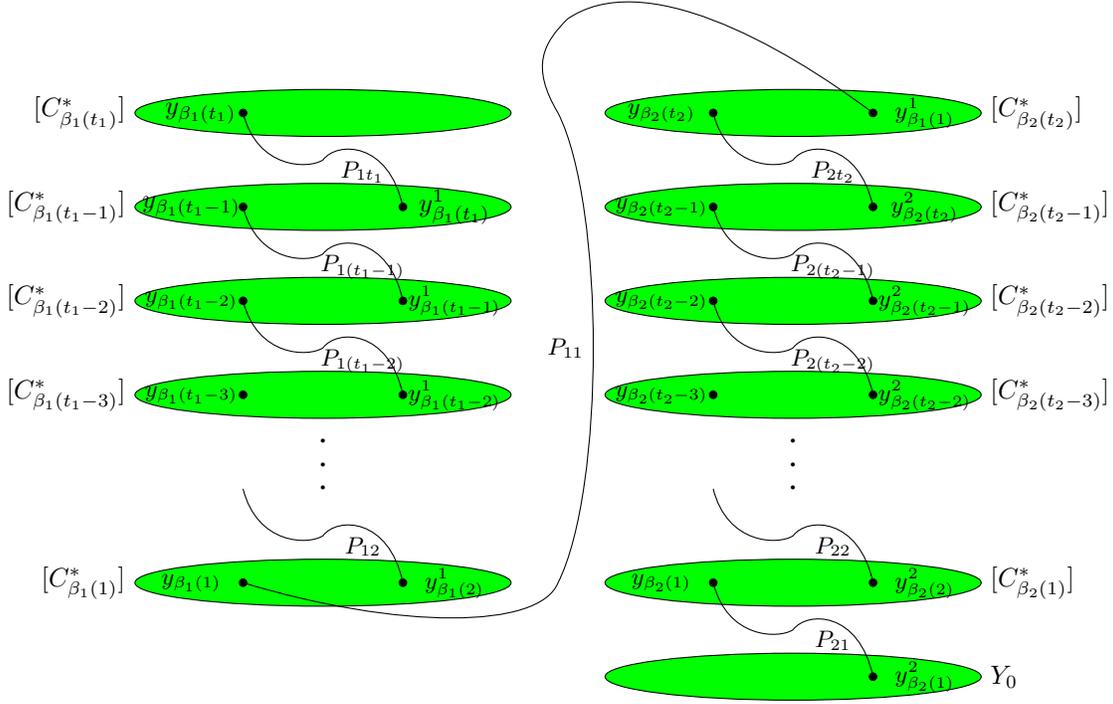

\subsection{Targets}

When we modify $W$, we also change $A$ so that it becomes a new basis $A'.$  The changes we make will be such that when $a_i \in A$ is replaced by an element $a_i' \in A'$ which we view as the {\it target} of $a_i.$   We will require that our targets $a_i'$ belong to a subset $T_i \subseteq C_i^*.$

\begin{definition} A {\bf target} for $A$ is a $k$-tuple of subsets $\T = (T_1,T_2, \dots ,T_k)$ where for all $i\in [k]$, $a_i \in T_i$ and $T_i \subseteq C_i^*.$\end{definition}

Let $K= (P_1, \dots ,P_t)$ be a path-tangle rooted at $Y_0$, where for all $i \in [t],$ $y_{\beta(i)} \xrightarrow [W] {P_i} y_{\beta(i)}'$.  Then $K$ is 
said to be {\bf consistent} with the target $\T$ if for $i=2, \dots ,t$, either $y_{\beta(i)}' \in Y_0$ or $y_{\beta(i)}' \in [T_{\beta(i')}]$ for some $i' < i.$

\subsection{The flat $\mathbf{F_0}$}

For the remainder of the paper, we shall assume that there are flats $F\in \frakF$ for which $\omega_W(F) < r_N(F).$  We shall assume that among such flats, $F_0$ is such a flat having least rank.

\begin{definition} Let $F_0 \in \frakF$ be a flat where 
\begin{itemize}
\item $\omega_W(F_0) < r_N(F_0).$ 
\item  For all flats $F\in \frakF$ where $r_N(F) < r(F_0),$ we have $\omega_W(F) = r_N(F).$
\end{itemize} 
\end{definition}  
In addition, we define a subsets $A^* \subseteq A$ and $A^{\spadesuit}\subseteq A$ as follows:

\begin{definition} $A^* = \{ a_i \in A\ \big| \ r_N(F_i) \le r_N(F_0) \ \mathrm{and}\ \omega_W(F_i) = r_N(F_i) \}$ and $A^{\spadesuit} = A - (F_0 \cup A^*).$\end{definition}

\noindent {\bf Note}: If $a_i \in F_0 - A^*,$ then $F_i = F_0.$ To see this, if $F_i \ne F_0,$ then $F_i \subset F_0.$  By our choice of $F_0$, we have $\omega_W(F_i) = r_N(F_i),$ implying that $a_i \in A^*,$ a contradiction. 

When we change the basis $A$ to a basis $A'$, we want to ensure that for all flats $F_i \in \frakF$ where $r_N(F_i) \le r_N(F_0),$  $|A' \cap F_i| \ge |A \cap F_i|.$  In particular, if $a_i \in A \cap (F_0 \cup A^*) = A - A^{\spadesuit}$ and $a_i$ is replaced by its target $a_i'$, then we want to ensure that $a_i' \in F_i.$ Otherwise, if $a_i \in A^{\spadesuit},$ then we allow $a_i'\not\in F_i.$  For this reason, we shall need to refine the target sets we use. 

\begin{definition} We say that $\T = (T_1,T_2, \dots ,T_k)$ is an $\mathbf{F_0}${\bf - target} if:
\begin{itemize}
\item For all $a_i \in A \cap (F_0 \cup A^*) = A - A^{\spadesuit}$, $T_i = F_i'.$
\item For all $a_i \in A^{\spadesuit},$ $T_i = C_i^*.$
\end{itemize}
\end{definition}

\begin{lemma}
Let $\T = (T_1,T_2, \dots ,T_k)$ be an $F_0$ - target.
\begin{itemize}
\item[a)] If $a_i \in A^{*}$ and $a_j \not\in F_i,$ then $C_j^* \cap T_i = \emptyset.$
\item[b)] If $a_i \in A^{*}$ and $T_i \cap T_j \ne \emptyset,$ then $F_i = F_j$, $a_j \in A^{*},$ and $T_i \cap C_j^* = T_i \cap T_j$.
\item[c)] If $a_i, a_j \in A^{*}$ and $(C_i^* - T_i)\cap T_j \ne \emptyset,$ then $F_i \subset F_j.$ 
\item[d)]  If $a_i, a_j  \in F_0 - A^*$, then $C_i^* \cap T_j = T_i \cap T_j.$
\end{itemize}\label{lem-target}
\end{lemma}

\begin{proof}
\begin{itemize}
\item[a)] Suppose $a_i \in A^{*}$ and $a_j \not\in F_i.$  By definition of $A^*,$ $\omega_W(F_i) = r_N(F_i).$  It follows that $F_i \subseteq H_j$ and thus $C_j^* \cap T_i = C_j^* \cap F_i' = \emptyset.$
\item[b)] Suppose $a_i \in A^{*}$ and $T_i \cap T_j \ne \emptyset.$  By a), $a_j\in F_i$ and thus $F_j \subseteq F_i$.  Consequently,  $a_j \in A^*.$   If $a_i \not\in F_j,$ then $T_i \cap T_j = \emptyset$ (by a) ).
Thus $a_i \in F_j$ and consequently $F_i = F_j.$  From this, it follows that $T_i \cap C_j^* = F_i' \cap C_j^* = F_i' \cap F_j' = T_i \cap T_j.$
\item[c)] Suppose $a_i, a_j \in A^{*}$ and $(C_i^* - T_i)\cap T_j \ne \emptyset.$  By a), $a_i \in F_j$ and hence $F_i \subseteq F_j.$  If $F_i = F_j$, then 
$C_i^* \cap T_j = C_i^* \cap F_j' = C_i^* \cap (C_j^*\cap F_j) = (C_i^*\cap F_i) \cap (C_j^*\cap F_j) = F_i' \cap F_j' = T_i \cap T_j.$  However, this contradicts the assumption that $(C_i^* - T_i)\cap T_j \ne \emptyset.$ 
Thus $F_i \subset F_j.$
\item[d)] Assume that $a_i, a_j  \in F_0 - A^*$.  We first observe that $F_i = F_j = F_0.$   Thus $T_i = F_i' = C_i^* \cap F_0$ and $T_j = F_j' = C_j^* \cap F_0.$
We now see that $C_i^* \cap T_j = C_i^* \cap (C_j^* \cap F_0) = (C_i^*\cap F_0) \cap (C_j^* \cap F_0) = T_i \cap T_j.$
\end{itemize}
\end{proof}

\begin{definition}Let $\T = (T_1, \dots ,T_k)$ be a target for $A$  and let $K = (P_1, \dots ,P_t)$ be a path-chain consistent with $\T$ and rooted at $Y_0$ where $\beta:[t] \rightarrow [k]$ is such that
$\forall i\in [t], y_{\beta(i)} \xrightarrow [W] {P_i} y_{\beta(i)}'$.  Then $K$ 
is said to be {\bf simple} if the paths $P_1, \dots ,P_t$ are vertex-disjoint and for $j=2, \dots, t,$ $\psi(y_{\beta(j)}') \in T_{\beta(j-1)} \backslash \bigcup_{j' <j-1}T_{\beta(j')}.$\end{definition}

\section{Choosing $W$ so that $\forall v \in V, \delta_{p+1}(v) \ge 0$}

We wish to choose $W= W_{p+1}$ in $G_p$ in such a way that $\forall v \in V,\ \delta_{p+1}(v) \ge 0.$  
To do this, we need to ensure that  $\forall v\in V, \ d_{p+1}(v) \le \eta-1.$
It follows by Lemma \ref{obs1.5} that for all vertices $v\in V,$ $\delta_{p+1}(v) = \delta_p(v) - \mathbf{1}_{\overline{V}_W}(v).$  
Given that for all $v\in V,\ \delta_p(v) \ge 0,$  it suffices to choose $W$ in such a way that $\forall v\in \overline{V}_W, \ \delta_p(v) >0.$
To achieve this goal, we will show that that $W$ can be changed incrementally using alternating paths.

\begin{proposition}
One can choose the matching $W$ in $G_{p}$ such that $\forall v\in \overline{V}_W, \ \delta_p(v) >0.$
\label{pro-chooseW1}
\end{proposition}

\begin{proof}
 We may assume that there are vertices 
$y_i\in \overline{V}_W$ such that 
$\delta_p(y_i) = 0.$  It suffices to show that one can construct a matching $W'$ for which $V_{W'}$ is a basis and $\overline{V}_{W'}$ has fewer vertices $v$ than $\overline{V}_W$ for which $\delta_p(v) = 0.$
To begin with, we may assume (without loss of generality) that $\delta_p(y_k) = 0.$   

\begin{claim}  Let  $Y_0 = [C_k^*]- y_k.$ There exists a path-tangle $K = (P_1, \dots ,P_t)$ rooted at $Y_0$ where for some function $\beta:[t] \rightarrow [k],\ \beta(t) =k,$ and $\forall i\in [t], \ y_{\beta(i)} \xrightarrow [W] {P_i} y_{\beta(i)}'.$ \label{claim-pathtangle}
\end{claim}

\begin{proof} By Observation \ref{obs-deltabound} we have $\delta_p(Y_0) = \delta_p(C_k^*)  \ge \eta.$
 By Lemma \ref{lem3}, there is an $Y_0$-path $P_1$ from some vertex $y_{i_1} \in \overline{V}_W$ to some vertex $y_{i_1}' \in \tilde{Y}_0$ (where $\delta_p(y_{i_1}') >0$).  If $y_{i_1} = y_k,$ then $K=P_1$ is the required path-tangle.  Thus we may assume that $y_{i_1} \ne y_k.$  Let $Y_1 = Y_0 \cup [C_{i_1}^*] = ([C_k^*] \cup [C_{i_1}^*]) - y_k.$  Then by Observation \ref{obs-deltabound}, $\delta_p(Y_1)  = \delta_p([C_{i_1}^*] \cup [C_k^*]) = \delta_p (C_{i_1}^* \cup C_{k}^*) \ge 2\eta$ and $Y_1 \cap \{ y_1, \dots ,y_k \} = \{ y_{i_1} \}.$  By Lemma \ref{obs6}, there is a $Y_1$-path $y_{i_2} \xrightarrow [W]{P_2} y_{i_2}'$ from a vertex $y_{i_2} \in \overline{V}_W - Y_1$ to a vertex $y_{i_2}' \in \tilde{Y}_1$ (where $\delta_p(y_{i_2}') >0$)  If $y_{i_2} = y_k,$ then $K=(P_1,P_2)$ is the required path-tangle.  Thus we may assume that $y_{i_2} \ne y_k.$  Continuing, suppose we have constructed sets $Y_0,Y_1, \dots ,Y_s$ and paths $P_1, \dots ,P_{s+1}$ where 
 \begin{itemize}
 \item For all $j\in[s],$ $Y_j = Y_0 \cup [C_{i_1}^*] \cup \cdots \cup [C_{i_j}^*]$, and $P_{j+1}$ is a $Y_j$-path where $y_{i_{j+1}} \xrightarrow [W]{P_{j+1}} y_{i_{j+1}}'$.
 \item For all $j\in [s],$ $y_{i_{j+1}} \in \overline{V}_W - Y_j$ and $y_{i_{j+1}}' \in \tilde{Y}_{j}$ (and $\delta_p(y_{i_{j+1}}') >0$).
 \item For all $j\in [s]$, $i_j \ne k.$
 \end{itemize}
 
 If $i_{s+1} =k,$ then $K = (P_1, \dots ,P_{s+1})$ is the desired path-tangle.  Suppose $i_{s+1} \ne k$ and let $Y_{s+1} = Y_0 \cup [C_{i_1}^*] \cup \cdots \cup [C_{i_{s+1}}^*].$
By Observation \ref{obs-deltabound}, $\delta_p(Y_{s+1}) = \delta_p \left( C_k^* \cup C_{i_1}^* \cup \cdots \cup C_{i_{s+1}}^* \right) \ge (s+2)\eta.$ Since $|Y_{s+1} \cap \{y_1, \dots ,y_k \}| = s+1,$
it follows from
Lemma \ref{obs6}, that there is a $Y_{s+1}$-path 
$y_{i_{s+2}} \xrightarrow [W]{P_{s+2}} y_{i_{s+2}}'$
from some vertex $y_{i_{s+2}} \in \overline{V}_W - Y_{s+1}$ to some vertex $y_{i_{s+2}}' \in \tilde{Y}_{s+1}$ (where $\delta_p( y_{i_{s+2}}') >0$). If $y_{i_{s+2}} = y_k,$ then $K=(P_1, \dots ,P_{s+2})$ is the desired path-tangle.  On the other hand, if $y_{i_{s+2}} \ne y_k,$ then the process continues with $s+2$ in place of $s+1.$
Since this process must eventually terminate, we must eventually end up with the desired path-tangle $K = (P_1, \dots ,P_t)$.
\end{proof}

Let $K = (P_1, \dots ,P_t)$ be the path-tangle described in Claim \ref{claim-pathtangle}.
Among such path-tangles, we may assume that $K$ contains a minimum number of paths.  We claim that $K$ is a simple path-chain.  If for some $j_1 < j_2,$ the paths $P_{j_1}$ and $P_{j_2}$ intersect, then it follows by Observation \ref{obs-alternating1} that $y_{i_{j_2}} \xrightarrow [W]{P} y_{i_{j_1}}'$ for some path $P$. In this case, $K' = (P_1, \dots ,P_{j_1-1}, P, P_{j_2+1}, \dots ,P_t)$ would be a path-tangle with fewer paths.  Thus by our choice of $K$, the paths $P_1, \dots ,P_t$ must be pairwise disjoint.  By a similar argument using minimality, one can show that for $j = 2, \dots ,t,$  $\psi(y_{i_j}') \in C_{i_{j-1}}^* \backslash \bigcup_{j' < j-1}C_{i_{j'}}^*.$  Thus $K$ is a simple path-chain.

For all $j \in [t],$ let $a_{i_j}' = \psi(y_{i_j}').$  Since $K$ is a simple path-chain, we have for $j = 2, \dots ,t, \  y_{i_j}' \in [C_{i_{j-1}}^*],$ and thus for $j = 2, \dots ,t, \ a_{i_j}' \in C_{i_{j-1}}^*.$

\begin{claim} $A' = A - \{ a_{i_1}, \dots ,a_{i_t} \} + \{ a_{i_1}', \dots ,a_{i_t}' \}$ is a basis in $N.$
\label{claim-pathtangle2}
\end{claim}  

\begin{proof} First we claim that $A^1 = A  - a_{i_{t-1}} + a_{i_t}'$ is a basis.  Given this is clear if $a_{i_t}' = a_{i_{t-1}},$  we may assume that $a_{i_t}' \ne a_{i_{t-1}}.$  Then $(A + a_{i_t}') \cap C_{i_{t-1}}^* = \{ a_{i_{t-1}}, a_{i_t}' \}$ and it follows that $A^1$ is a basis in $N.$  Suppose now that for some $s\ge 1$, $A^s = A - \{ a_{i_{t-s}}, \dots ,a_{i_{t-1}} \} + \{ a_{i_{t-s+1}}', \dots ,a_{i_{t}}' \}$ is a basis.
Assuming $s < t-1$, we claim that $A^{s+1} = A^s - a_{i_{t-s-1}} + a_{i_{t-s}}'$ is a basis.  We may assume $a_{i_{t-s-1}} \ne a_{i_{t-s}}'$ (for otherwise, $A^{s+1} = A^s$).  
Since $K$ is a simple path-chain, we have for $j= t-s+1, \dots ,t,$ $a_{i_j}' \not\in C_{i_{t-s-1}}^*.$  Consequently, $(A^s + a_{i_{t-s}}') \cap C_{i_{t-s-1}}^* = \{ a_{i_{t-s-1}}, a_{i_{t-s}}' \},$ and thus $A^{s+1}$ is a basis.
By induction, we see that $A^{t-1}  =  A - \{ a_{i_1}, \dots ,a_{i_{t-1}} \} + \{ a_{i_2}', \dots ,a_{i_t}' \}$ is a basis in $N.$  Since $K$ is a simple path-chain, it follows that for $j=2, \dots ,t,$ $a_{i_j}' \not\in C_k^*$ and this implies that $(A^{t-1} + a_{i_1}')\cap C_k^* = \{ a_k, a_{i_1}' \}.$  It now follows that $A' = A^{t-1} - a_k + a_{i_1}$ is a basis.
\end{proof}

Finally, let $W' = W \triangle E(P_1 \cup \cdots \cup P_t).$  Then $\overline{V}_{W'} = \overline{V}_W - \{ y_{i_1}, \dots , y_{i_t} \} + \{ y_{i_1}', \dots , y_{i_t}' \}$ and $A' = \psi(\overline{V}_{W'}).$
Since $\forall j\in [t],\ \delta_p(y_{i_j}') >0,$ we see that $\overline{V}_{W'}$ has fewer vertices $v$ than $\overline{V}_W$ for which $\delta_p(v) = 0.$  Furthermore, since $A'$ is a basis in $N$, it follows that $V_{W'}$ is a basis in $M.$  
 \end{proof}

\section{Choosing $W$ so that for all flats $F\in \frakF$, $\omega_W(F) = r_N(F)$} 
In light of Proposition \ref{pro-chooseW1}, we may assume that $\forall v\in \overline{V}_W, \ \delta_p(v) >0.$  Our goal in this section is to prove the following theorem:

\begin{theorem} The matching $W$ can be chosen so that
\begin{itemize}
 \item[i)]$V_W$ is a basis 
 \item[ii)] $\forall v\in V, \ \delta_{p+1}(v) \ge 0$ and 
\item[iii)] $\forall F\in \frakF$, $\omega_W(F) = r_N(F)$.
\end{itemize}\label{the-augment}
\end{theorem}    

\begin{definition}
Let $W'$ and $W''$ be matchings in $G_p$ for which $V_{W'}$ and $V_{W''}$ are bases. We define a partial order $\preceq$ where $W' \preceq W''$ if the following holds:  for all  $F \in \frakF,$ where $r_N(F) \le r_N(F_0)$, if $\omega_{W'}(F) = r_N(F),$ then $\omega_{W''}(F) = r_N(F).$\end{definition}

We will introduce an procedure by which we alter $W$ and construct a matching $W'$ in $G_p$ for which
\begin{itemize}
\item $V_{W'}$ is a basis of $M$.
\item $\forall v\in \overline{V}_{W'}, \ \delta_p(v) >0,$
\item $W \preceq W'$
\item $\omega_{W'}(F_0) > \omega_W(F_0).$
\end{itemize}
Repeated iterations of the above procedure will eventually yield a matching $W'$ for for which $W \preceq W'$ and $\omega_{W'}(F_0) = r_N(F_0).$ 
If there are flats $F\in \frakF$ for which $\omega_{W'}(F) < r_N(F),$ then we choose such a flat $F_0'$ of least rank and repeat the procedure with $F_0'$ in place of $F_0$ and $W'$ in place of $W.$  Eventually we achieve a matching $W'$ for which $V_{W'}$ is a basis and for which $\forall F\in \frakF, \ \omega_{W'}(F) = r_N(F).$
Thus to prove Theorem \ref{the-augment}, it suffices to show that the procedure works.   We shall accomplish this by choosing a specific simple path-chain.

We shall need a lemma pertaining to the circuit exchanges for a collection of circuits.
Suppose $\C = \{ C_1, \dots ,C_{\ell} \}$ is a collection of circuits in a matroid.  Let $x_1, \dots ,x_{\ell}$ be distinct elements in the matroid where for $i = 1, \dots ,\ell,$ $x_i \in C_i.$  Let $\overrightarrow{D} = \overrightarrow{D}(\C, x_1, \dots ,x_\ell)$ be the directed graph on vertices $z_1, \dots ,z_{\ell}$ where $z_i \rightarrow z_j$ if $x_j \in C_i.$ 

\begin{lemma}
Suppose $\overrightarrow{D}$ is acyclic.  
Then for $i = 1, \dots ,\ell$, there is a circuit $C$ where $C\subseteq \bigcup_{j=1}^\ell C_j$ and $C \cap \{ x_1, \dots ,x_{\ell} \} = \{ x_i \}.$\label{lem-acyclic}
\end{lemma}

\begin{proof}
By induction on $\ell$.  When $\ell = 1,$ the lemma is clearly true.   Assume the lemma holds for collections of fewer than $\ell$ circuits and let $\C$ and $\{ x_1, \dots , x_{\ell} \}$ be as described above.  Since $\overrightarrow{D}$ is acyclic, it has at least one 
source vertex and without loss of generality we may assume $z_{\ell}$ is such a vertex.  Since for all $j = [\ell-1],$ $z_{\ell} \not\rightarrow z_j$, it follows that $C_{\ell} \cap \{ x_1, \dots ,x_{\ell} \} = \{ x_{\ell} \}.$  Thus the assertion in the lemma is true when $i=\ell.$  Let $i \in [\ell -1].$  We will show that there is a circuit $C \subseteq \bigcup_{j=1}^\ell C_j$ where $C\cap \{ x_1, \dots ,x_{\ell} \} = \{ x_i \}.$  We define circuits $C_j',\ j = 1, \dots ,\ell-1$ as follows:  if $z_j \rightarrow z_{\ell},$ then let $C_j'$ be a circuit where
$C_j' \subseteq C_j \cup C_{\ell} - x_{\ell}$ and $x_j \in C_j';$ such a circuit exists by the strong circuit elimination axiom.  If $z_j \not\rightarrow z_{\ell},$ then let $C_j' = C_j.$
Let $\C' = \{ C_1', \dots ,C_{\ell}' \}$ and let $\overrightarrow{D}' = \overrightarrow{D}(\C',x_1, \dots ,x_{\ell -1} \}$ be the corresponding directed graph on vertices $z_1',\dots ,z_{\ell-1}'$.  We claim that for all $i_1, i_2 \in [\ell-1],$ if $z_{i_1}' \rightarrow z_{i_2}',$ then $z_{i_1} \rightarrow z_{i_2}.$ To see this, 
suppose that $z_{i_1}' \rightarrow z_{i_2}'$ for some $i_1, i_2 \in [\ell-1].$  Then $x_{i_2} \in C_{i_1}'.$
If $C_{i_1}' = C_{i_1},$ then $x_{i_2} \in C_{i_1}$ and hence $z_{i_1} \rightarrow z_{i_2}.$  If $C_{i_1}' \ne C_{i_1},$ then $C_{i_1}' \subseteq C_{i_1} \cup C_{\ell} - x_\ell.$  Since $x_{i_2} \not\in C_{\ell},$ it follows that $x_{i_2} \in C_{i_1}$ and hence $z_{i_1} \rightarrow z_{i_2}.$  This proves our claim and in particular, $\overrightarrow{D}'$ must be acyclic since $\overrightarrow{D}$ is.  By assumption, there is a circuit $C' \subseteq \bigcup_{j=1}^{\ell-1}C_j'$ where $C' \cap \{ x_1, \dots ,x_{\ell-1} \} = \{ x_i \}.$  Thus $C_j' \subseteq \bigcup_{j=1}^\ell C_j$ and moreover, $C' \cap \{ x_1, \dots ,x_{\ell} \} = \{ x_i \}$ since $x_\ell \not\in \bigcup_{j=1}^{\ell-1}C_j'$.  Thus the lemma holds for $\C$ and this completes the induction.
\end{proof}

Let $A_0 =  F_0 -  \mathrm{cl}_N(F_0\cap A)$ and let $\T =  (T_1, \dots ,T_k)$ be an $F_0$-target and let $Y_0 = [A_0].$ 
We shall assume that $K= (P_1, \dots ,P_t)$ is a simple path-chain consistent with $\T$ rooted at $Y_0$ where for $\beta:[t] \rightarrow [k]$ we have $\forall i\in [t],\ y_{\beta(i)} \xrightarrow [W] {P_i} y_{\beta(i)}'.$  For all $i\in [t],$ let $a_{\beta(i)}' = \psi(y_{\beta(i)}').$  Our first task is to find conditions for $K$ such that when we exchange the elements $a_{\beta(i)}, i \in [t],$ in $A$ for the elements $a_{\beta(i)}',i \in [t]$, we achieve a new basis.

Recall that by the definition of an $F_0$-target, we have that $\forall a_i \in A^\spadesuit, \ T_i = C_i^*$ and $\forall a_i \in A - A^\spadesuit,\ T_i = F_i' = F_i \cap C_i^*.$ 

\begin{proposition}
Suppose $a_{\beta(t)} \in A^{\spadesuit}$ and $a_{\beta(1)}' \in A_0 \cap T_{\beta(t)}$ and suppose there is no pair $j, j' \in [t]$, $0 \le j-j' < t-1,$ for which $a_{\beta(j)} \in A^{\spadesuit}$ and 
$a_{\beta(j')}' \in A_0 \cap T_{\beta(j)}.$ Then $A - \{ a_{\beta(1)}, \dots ,a_{\beta(t)} \} + \{ a_{\beta(1)}', \dots ,a_{\beta(t)}' \}$ is a basis for $N$. 
\label{pro-SimpleTangle}
\end{proposition}

\begin{proof}
Let $\C^* = \{ C_{\beta(1)}^*, \dots ,C_{\beta(t)}^* \}$ and for $i = 1, \dots, t,$ let $x_i = a_{\beta(i+1)}',$ where $a_{\beta(t+1)}' = a_{\beta(1)}'.$  We shall define a directed graph $\overrightarrow{D} = \overrightarrow{D}(\C^*, x_1, \dots ,x_t)$ in a similar fashion as was done prior to Lemma \ref{lem-acyclic}: Let $V_{\overrightarrow{D}} = \{ z_1, \dots ,z_t \}.$ For all $i,j \in [t],$ where $i \ne j,$ $z_i \rightarrow z_j$ if $x_j \in C_{\beta(i)}^*.$ 

\begin{noname}
The directed graph $\overrightarrow{D}$ is acyclic.\label{noname0.5}
\end{noname}

\begin{proof}
To begin with, we make some useful observations.  

\begin{itemize}
\item[i)]  If $i <j$ and $z_i \rightarrow z_j$, then $a_{\beta(i)} \in A - A^\spadesuit.$
\begin{proof}
Suppose $i <j$ and $z_i \rightarrow z_j$.  We have $x_j \in C_{\beta(i)}^* \cap T_{\beta(j)}.$  Since $K$ is a simple path-chain, $x_j  \not\in T_{\beta(i)}$ and it follows that $T_{\beta(i)} \ne C_{\beta(i)}^*.$  Thus $a_{\beta(i)} \in A- A^{\spadesuit}$
\end{proof}
\item[ii)] If $i <j,$ $z_i \rightarrow z_j$, and $a_{\beta(i)} \in F_0 - A^*,$ then $a_{\beta(j)} \in A^{\spadesuit}.$
\begin{proof}
Assume that for some $i<j$, $z_i \rightarrow z_j$ where $a_{\beta(i)} \in F_0 - A^*.$  We have $x_j \in (C_{\beta(i)}^* - T_{\beta(i)})\cap T_{\beta(j)}.$  If $a_{\beta(j)} \in F_0 - A^*$, then 
it follows by Lemma \ref{lem-target} d) that $C_{\beta(i)}^* \cap T_{\beta(j)} = T_{\beta(i)} \cap T_{\beta(j)},$ which is a contradiction. Thus  $a_{\beta(j)} \not\in F_0 - A^*.$   If $a_{\beta(j)} \in A^*,$
then by Lemma \ref{lem-target} a), $C_{\beta(i)}^* \cap T_{\beta(j)} = \emptyset,$ a contradiction.  Thus $a_{\beta(i)} \in A^{\spadesuit}.$
\end{proof}
\item[iii)]  Suppose $z_i \rightarrow z_j$ where $\{ a_{\beta(i)}, a_{\beta(j)} \} \subseteq A^*.$ Then $F_{\beta(i)} \subseteq F_{\beta(j)}.$  
Moreover, if $i<j,$ then $F_{\beta(i)} \subset F_{\beta(j)}.$
\begin{proof}
 We have $x_j \in C_{\beta(i)}^* \cap T_{\beta(j)}.$  It follows by
Lemma \ref{lem-target} a) that $a_{\beta(i)} \in F_{\beta(j)}$ and thus $F_{\beta(i)} \subseteq F_{\beta(j)}.$  
If $i<j$, then $x_j \in (C_{\beta(i)}^*- T_{\beta(i)}) \cap T_{\beta(j)}.$ It follows by
Lemma \ref{lem-target} c) that $F_{\beta(i)} \subset F_{\beta(j)}.$
\end{proof} 
\item[iv)]  If $i<j,$ $z_j \rightarrow z_i$, and $a_{\beta(j)} \in A^{\spadesuit}$, then $a_{\beta(i)} \in A^{\spadesuit}.$
\begin{proof}
Assume that $i<j,$ $z_j \rightarrow z_i$, and $a_{\beta(j)} \in A^{\spadesuit}.$  Suppose $a_{\beta(i)} \in A - A^{\spadesuit}$.   Since $C_{\beta(j)}^* \cap T_{\beta(i)}\ne \emptyset,$ it follows by Lemma \ref{lem-target} a) that $a_{\beta(i)} \in F_0 - A^*$ and hence $F_{\beta(i)} = F_0.$  However, we now have that $a_{i+1}' = x_i \in T_{\beta(i)} \cap C_{\beta(j)}^* = F_{\beta(i)}' \cap C_{\beta(j)}^* \subseteq A_0 \cap T_{\beta(j)}.$  Since $a_{\beta(j)} \in A^{\spadesuit}$ and $j- (i+1) \le t-2,$ this contradicts one of our assumptions.  Therefore $a_{\beta(i)} \in A^{\spadesuit}$.
\end{proof} 
\end{itemize}

Suppose to the contrary that $\overrightarrow{D}$ contains a directed cycle $z_{i_1} \rightarrow z_{i_2} \rightarrow \cdots \rightarrow z_{i_\ell} \rightarrow z_{i_1}.$
We may assume that $i_1 = \min \{ i_j\ \big| \ j=1, \dots ,\ell \}.$  Since $i_1 < i_2$ and $z_{i_1} \rightarrow z_{i_2}$, it follows by i) that $a_{\beta(i_1)} \in A - A^{\spadesuit}.$
If $a_{\beta(i_2)} \in A^{\spadesuit},$ then it follows by i) that $i_2 > i_3.$  Furthermore, iv) implies that $a_{\beta(i_3)} \in A^{\spadesuit}.$ Applying the same reasoning we obtain that
that $i_2 > i_3 > \cdots > i_\ell > i_1$ and $a_{\beta(i_j)} \in A^{\spadesuit},\ j = 1 ,\dots \ell,$ yielding a contradiction
since $a_{\beta(i_1)} \in A - A^{\spadesuit}.$  It follows that $a_{\beta(i_2)} \in A - A^\spadesuit.$  In fact, the above reasoning shows that $\forall j\in [\ell],$ $a_{\beta(i_j)} \in A - A^{\spadesuit}$.  Suppose $a_{\beta(i_1)} \in F_0 - A^*.$ Then it follows by ii) that $a_{\beta(i_2)} \in A^{\spadesuit} ,$ yielding a contradiction. Thus $a_{\beta(1)} \in A^*.$  We observe that if $a_{\beta(2)} \in A^*,$
then it follows by iii) that $F_{\beta(1)} \subset F_{\beta(2)}.$  We claim that there exists $2 \le \ell' < \ell$ such that $\forall j \in [\ell'-1], \ a_{\beta(i_j)}\in A^*$ and $a_{\beta(i_{\ell'})} \not\in A^*.$  
For if no such $\ell'$ existed, then we would have $\forall j\in [\ell], \ a_{\beta(i_j)}\in A^*$ and consequently by iii), $F_{\beta(i_1)} \subset F_{\beta(i_2)} \subseteq \cdots \subseteq F_{\beta(i_\ell)} 
\subseteq F_{\beta(i_1)},$ a contradiction.
Since $a_{\beta(i_{\ell'})} \not\in A^*,$ it follows that $a_{\beta(i_{\ell'})} \in F_0-A^*.$
If $i_{\ell'} < i_{\ell'+1},$ then it follows by ii) that $a_{\beta(i_{\ell'+1})} \in A^{\spadesuit},$ a contradiction.  Thus $i_{\ell'+1} < i_{\ell'}.$   
Since $x_{\ell'+1} \in C_{\beta(i_{\ell'})}^* \cap T_{\beta(i_{\ell'+1})},$ it follows by Lemma \ref{lem-target} a), that $a_{\beta(i_{\ell' +1})} \not\in A^*.$  
Thus  $a_{\beta(i_{\ell'+1})} \in F_0 -A^*.$  Following the same reasoning, we obtain that 
$i_{\ell'} > i_{\ell'+1} > \dots > i_\ell > i_1$, and $\forall j\in \{ 1, \ell', \ell'+1, \dots , \ell\},$ $a_{\beta(i_j)} \in F_0 - A^*$.  This gives a contradiction because $a_{\beta(i_1)} \in A^*.$ 
\end{proof} 

It follows that by (\ref{noname0.5}) and Lemma \ref{lem-acyclic} that for all $i = 1, \dots ,t,$ there is a cocircuit $C^* \subseteq \bigcup_{j=1}^tC_{\beta(j)}^*$ for which $C^* \cap \{ x_1, \dots ,x_t \} = \{ x_i \}.$  In fact, for any nonempty subset 
$I \subseteq [t]$ we have that for all $i\in I$, there is a cocircuit $C^* \subseteq \bigcup_{j\in I}C_{\beta(j)}^*$ for which $C^* \cap \{ x_j \ \big| \ j\in I \} = \{ x_i \}.$
This observation will be used below.

\begin{noname}
For all $2 \le s \le t-1,$ there is a cocircuit $C^*$ in $N$ such that $$C^* \cap \left(  (A - \{ a_{\beta(s)}, \dots ,a_{\beta(t-1)} \}) \cup \{ a_{\beta(s)}', \dots ,a_{\beta(t)}' \}  \right)   \subseteq \{ a_{\beta(s-1)}, a_{\beta(s)}' \}.$$
\label{noname1.5}
\end{noname}

\begin{proof}
By the above, there is a cocircuit $C^*$ in $N$ where $C^* \subseteq \bigcup_{j=s-1}^{t-1} C_{\beta(j)}^*$ where $C^*\cap \{ x_{s-1}, \dots ,x_{t-1} \} = \{ x_{s-1} \}.$ That is, 
$C^* \cap \{ a_{\beta(s)}', a_{\beta(s+1)}', \dots ,a_{\beta(t)}' \} = \{ a_{\beta(s)}' \}.$
It follows that 
$C^*$ is a cocircuit for which $C^* \cap \left( (A - \{ a_{\beta(s)}, \dots ,a_{\beta(t-1)}) \cup \{ a_{\beta(s)}',  \dots ,a_{\beta(t)}' \}  \right)  \subseteq \{ a_{\beta(s-1)}, a_{\beta(s)}' \}.$
\end{proof}

\begin{noname}
There is a cocircuit $C^*$ in $N$ such that $C^* \cap \left( (A - \{ a_{\beta(1)}, \dots ,a_{\beta(t-1)} \} ) \cup \{ a_{\beta(1)}', \dots ,a_{\beta(t)}' \} \right) \subseteq \{ a_{\beta(t)}, a_{\beta(1)}' \}.$
\label{noname1.55}
\end{noname}

\begin{proof}
Again by the observation above,
there is a cocircuit $C^*$ in $N$ where $C^* \subseteq \bigcup_{j=1}^{t} C_{\beta(j)}^*$ and $C^* \cap \{ x_1, \dots ,x_t \} = \{ x_t \}.$  That is,
$C^* \cap \{ a_{\beta(1)}', a_{\beta(2)}', \dots ,a_{\beta(t)}' \} = \{ a_{\beta(1)}' \}.$
Thus  $C^*$ is a cocircuit for which $C^* \cap \left( (A - \{ a_{\beta(1)}, \dots ,a_{\beta(t-1)} \} ) \cup \{ a_{\beta(1)}', \dots ,a_{\beta(t)}' \} \right) \subseteq \{ a_{\beta(t)}, a_{\beta(1)}' \}.$
\end{proof}

\begin{noname}
$A - \{ a_{\beta(1)}, \dots ,a_{\beta(t-1)} \} + \{ a_{\beta(2)}', \dots ,a_{\beta(t)}' \}$ is a basis for $N.$\label{noname2}
\end{noname}

\begin{proof} We will first show that $A^1 = A - a_{\beta(t-1)} + a_{\beta(t)}'$ is a basis of $N.$ 
If $a_{\beta(t-1)} = a_{\beta(t)}',$ then $A^1 = A,$ and hence it is a basis.  Thus we may assume $a_{\beta(t-1)} \ne a_{\beta(t)}'.$  We see that $(A + a_{\beta(t)}')\cap C_{\beta(t-1)}^* = \{ a_{\beta(t-1)}, a_{\beta(t)}' \}$ and hence it follows that $A^1$ is a basis.  

Suppose that for some $2 \le s \le t-1$, $A^{t - s} = A - \{ a_{\beta(s)},  \cdots ,a_{\beta(t-1)} \} + \{ a_{\beta(s+1)}',  \cdots  ,a_{\beta(t)}' \}$ is a basis.  We claim that $A^{t-s+1} = A^{t-s} - a_{\beta(s-1)} + a_{\beta(s)}'$ is a basis.  If $a_{\beta(s)}' = a_{\beta(s-1)},$ then $A^{t-s+1} = A^{t-s}$ and the assertion is true.  Thus we may assume that 
$a_{\beta(s)}' \ne a_{\beta(s-1)}.$  By (\ref{noname1.5}), there is a cocircuit $C^*$ in $N$ for which $C^* \cap (A^{s-t} + a_{\beta(s)}') \subseteq \{ a_{\beta(s-1)}, a_{\beta(s)}' \}$.  Clearly we must have equality since the circuit $C$ in $A^{s-t} + a_{\beta(s)}'$ (containing $a_{\beta(s)}'$) also contains $a_{\beta(s-1)}$ (because $C_{\beta(s-1)}^* \cap C = \{ a_{\beta(s-1)}, a_{\beta(s)}'\}$) and $|C \cap C^*| \ge 2.$   Thus it follows that $A^{t-s+1}$ is a basis.
 Continuing inductively, we see that $A^{t-1} = A - \{ a_{\beta(1)}, \dots ,a_{\beta(t-1)} \} + \{ a_{\beta(2)}', \dots ,a_{\beta(t)}' \}$ is a basis of $N$.
\end{proof}
 
Finally, we shall show that
$A^t = A - \{ a_{\beta(1)}, \dots ,a_{\beta(t)} \} + \{ a_{\beta(1)}', \dots ,a_{\beta(t)}' \}$ is a basis for $N.$  By (\ref{noname2}), $A^{t-1} = A - \{ a_{\beta(1)}, \dots ,a_{\beta(t-1)} \} + \{ a_{\beta(2)}', \dots ,a_{\beta(t)}' \}$ is a basis for $N.$  By (\ref{noname1.55}), there is a cocircuit $C^*$ such that
$C^* \cap \left( (A - \{ a_{\beta(1)}, \dots ,a_{\beta(t-1)} \} ) \cup \{ a_{\beta(1)}', \dots ,a_{\beta(t)}' \} \right) \subseteq \{ a_{\beta(t)}, a_{\beta(1)}' \}.$  In fact, equality must hold since $A^{t-1}$ is a basis and hence $A^{t-1} + a_{\beta(1)}'$ has a unique circuit $C$ which contains $a_{\beta(1)}'$ and $C \cap C_{\beta(t)}^* = \{ a_{\beta(t)}, a_{\beta(1)}' \}.$  
Thus $C^* \cap A^{t-1} = \{ a_{\beta(t)}, a_{\beta(1)'} \}.$
It follows that $A^t = A^{t-1} - a_{\beta(t)} + a_{\beta(1)}'$ is a basis.
\end{proof}

\subsection{Finding a path-chain}

Our goal is to show that $W$ can be altered through through successive changes using alternating paths in such a way that if $W'$ is the resulting matching, then $V_{W'}$ is a basis for which the basis $A' = \psi(\overline{V}_{W'})$ has more elements in $F_0$ than $A$ does.  Furthermore,  we want $A'$ to have the property that for all flats $F\in \frakF$ having rank at most that of $F_0$,  $|A' \cap F| \ge |A \cap F|.$  Proposition \ref{pro-SimpleTangle} tells us which path-chains will guarantee that $A'$ is a basis.  In particular, we want to construct a path-chain $K = (P_1, \dots ,P_t)$, consistent with an $F_0$-target $\T = (T_1, \dots ,T_k)$ where for $\beta:[t] \rightarrow [k],\ \forall i \in [t],\ y_{\beta(i)} \xrightarrow[W]{P_{\beta(i)}} y_{\beta(i)}',$ $a_{\beta(t)} \in A^\spadesuit$ and $a_{\beta(1)}' = \psi(y_{\beta(1)}') \in A_0 \cap T_{\beta(t)}.$  Furthermore, we want to choose a minimal such path-chain.   

As before, let $A_0 = F_0 - \mathrm{cl}_N(F_0\cap A)$ and $Y_0 = [A_0].$
Let $\T = (T_1, \dots ,T_k)$ be an $F_0$-target.   As a first step in finding the desired path-chain, we will prove the following key lemma:

\begin{lemma}
Suppose $\eta > (k-1)^2 + k(k-1)^2 = (k+1)(k-1)^2.$
Let $A^1 \subseteq A^\spadesuit$ and let $A^2 \subseteq A  - A^{\spadesuit}.$  If $A_0 \not\subseteq \bigcup_{a_i \in A^1}T_i,$ then
$\delta_p \left( A_0 \cup \bigcup_{a_i \in A^1 \cup A^2}T_i \right) \ge (|A^1| +|A^2|+1)\eta - (k+1)(k-1)^2.$ \label{lem-needed1}
\end{lemma}

\begin{proof}
Let $A^{21} = A^2 - A^{*}$ and let $A^{22} = A^2 \cap A^{*}$.  By Lemma \ref{lem-B4} and the fact that $\forall a_i \in A^{22},\ T_i = F_i',$ we have 
\begin{equation}
\delta_p \left( \bigcup_{a_i \in A^{22}}T_i \right) = \delta_p \left( \bigcup_{a_i \in A^{22}}F_i'' \right)
 \ge |A^{22}|\eta - k(k-1)^2.\label{clm3eq1}
\end{equation}
Let $F = \bigcap_{a_i \in A^{1}}H_i.$  Then $r_N(F) \le k - |A^1|.$  Since $\forall a_i \in A^1,\ T_i = C_i^*,$ we also see that 
$$\delta_p\left( \bigcup_{a_i\in A^{1}} T_i \right) = \delta_p\left( \bigcup_{a_i\in A^{1}} C_i^* \right) = \delta_p(N) -\delta_p(F) = k\eta - \delta_p(F).$$
Let $\varepsilon_1\ge 0$ be such that $\delta_p(F) = r_N(F)\eta - \varepsilon_1.$ It follows by the previous equation and the inequality $r_N(F) \le k - |A^1|$ that
\begin{equation}
\delta_p\left( \bigcup_{a_i\in A^{1}} T_i \right) =
(k - r_N(F))\eta + \varepsilon_1 \ge |A^{1}|\eta + \varepsilon_1.\label{clm3eq2}
\end{equation}

 Let $F_0' = F\cap F_0$ and let $F_0'' = \mathrm{cl}(A \cap F_0).$  We note that $\forall a_i \in A^1,\ A\cap F_0 \subseteq H_i.$  Thus $A\cap F_0 \subseteq F$ and hence $F_0'' \subseteq F_0'.$   Since $\delta_p(F_0) \ge r_N(F_0)\eta - (k-1)^2$ and $\delta_p(F) = r_N(F)\eta - \varepsilon_1,$
 it follows by Lemma \ref{lem-B1} that 
 \begin{equation}
 \delta_p(F_0') \ge r_N(F_0')\eta - (k-1)^2 - \varepsilon_1.\label{clm3eq2.5}
 \end{equation}    
 Since by assumption, $A_0  \not\subseteq \bigcup_{a_i \in A^1}T_i = \bigcup_{a_i \in A^1}C_i^* = E(N) - F$, it follows that $F_0' \cap A_0 \ne \emptyset$ and hence
$r_N(F_0') > r_N(F_0'') = \omega_W(F_0).$  Note that for all $a_i \in A^{21},$ $F_i = F_0$ and thus $T_i = F_i' = C_i^* \cap F_0.$  For all $a_i \in A^{21},$  
let
$T_i' = T_i \cap F_0'' = F_i' \cap F_0'' = C_i^* \cap F_0''.$   Let $\delta_p(F_0'') = r_N(F_0'')\eta - \varepsilon_2.$
Then by (\ref{clm3eq2.5}) we have
\begin{align*} 
\delta_p(A_0 \cap F_0') &= \delta_p(F_0') - \delta_p(F_0'') \ge r_N(F_0')\eta - (k-1)^2 - \varepsilon_1 - r_N(F_0'')\eta + \varepsilon_2\\ &= (r_N(F_0') - r_N(F_0'')\eta
- (k-1)^2 - \varepsilon_1 + \varepsilon_2.\numberthis \label{clm3eq3}
\end{align*}
Furthermore, since $r_N\left(F_0'' \cap \bigcap_{a_i \in A^{21}}H_i \right) \le r_N(F_0'') - |A^{21}|$ and\\  $\delta_p(F_0'') = r_N(F_0'')\eta - \varepsilon_2$ we have
\begin{equation}
\delta_p\left( \bigcup_{a_i \in A^{21}}T_i' \right) =  \delta_p(F_0'') - \delta_p \left( F_0'' \cap \bigcap_{a_i \in A^{21}}H_i \right)
\ge \delta_p(F_0'') - (r_N(F_0'') - |A^{21}|)\eta
= |A^{21}|\eta - \varepsilon_2.\label{clm3eq4}
\end{equation}

Thus it follows by (\ref{clm3eq2}), (\ref{clm3eq3}) and (\ref{clm3eq4}) that
\begin{align*}
\delta_p\left(A_0 \cup \bigcup_{a_i \in (A^1 \cup A^{21})}T_i \right) &= \delta_p\left(\bigcup_{a_i \in A^{1}}T_i\right) + \delta_p(A_0 \cap F_0') + \delta_p\left( \bigcup_{a_i \in A^{21}}T_i' \right)\\
&\ge |A^{1}|\eta + \varepsilon _1+ 
(r_N(F_0') - r_N(F_0'')\eta
- (k-1)^2 - \varepsilon_1 + \varepsilon_2 + |A^{21}|\eta - \varepsilon_2\\
&= (|A^{1}| + |A^{21}|)\eta + (r_N(F_0') - r_N(F_0'')\eta - (k-1)^2\\
&\ge (|A^1|+|A^{21}| +1)\eta - (k-1)^2. \numberthis \label{clmeq5}
\end{align*}
It follows by Lemma \ref{lem-target} a) that for all $a_i \in A^1 \cup A^{21}$ and $a_j \in A^{22}$,  $T_i \cap T_j = \emptyset$.  
By Lemma \ref{lem-B4} we have $\delta_p\left(\bigcup_{a_i \in A^{22}}T_i \right) \ge |A^{22}|\eta - k(k-1)^2.$  Thus we have

 \begin{equation}\delta_p\left(A_0 \cup \bigcup_{a_i \in (A^1 \cup A^2)}T_i \right) \ge (|A^1|+|A^{21}|+1)\eta - (k-1)^2 + |A^{22}|\eta - k(k-1)^2 = (|A^1| + |A^2| +1)\eta - (k+1)(k-1)^2.
\end{equation}
This completes the proof of the lemma.
\end{proof}
We can now use the above lemma to find the desired path-chain.

\begin{proposition}
Suppose that $\eta > (k-1)^2 + k(k-1)^2 = (k+1)(k-1)^2.$   Then there is an function $\gamma: [s] \rightarrow [k]$ and a
path-chain consistent with $\T$, $K' = ( P_{\gamma(1)}, \dots ,P_{\gamma(s)}),$ where 

i) \ \ $\forall i \in [s],\ y_{\gamma(i)} \xrightarrow[W]{P_{\gamma(i)}} y_{\gamma(i)}'$  \ and \ ii) \ \ $a_{\gamma(s)} \in A^\spadesuit$ and $a_{\gamma(1)}' = \psi(y_{\gamma(1)}') \in A_0 \cap T_{\gamma(s)}.$

\label{pro-pathchain}
\end{proposition}

\begin{proof}
By Lemma \ref{obs-lowerbound}, $\delta_p(F_0) \ge r_N(F_0)\eta -(k-1)^2.$  As before, let $F_0'' = \mathrm{cl}(F_0\cap A)$.  We have that
$\delta_p(F_0'') \le r_N(F_0'')\eta = \omega_W(F_0)\eta.$  Given that $\omega_W(F_0'') < r_N(F_0),$ it follows that 
$$\delta_p(Y_0) = \delta_p(A_0) = \delta_p(F_0) - \delta_p(F_0'') \ge (r_N(F_0) - \omega_W(F_0))\eta - (k-1)^2 \ge \eta - (k-1)^2 > 0.$$ 
Since $\delta_p(Y_0) >0,$ it follows by Lemma \ref{lem3} that there is a $Y_0$-path $P_{i_1}$ from some vertex $y_{i_1} \in \overline{V}_W$ to a vertex $y_{i_1}' \in \tilde{Y}_0.$  Let $a_{i_1}' = \psi(y_{i_1}'),$ noting that $\delta_p(a_{i_1}') > 0.$   Let $Y_1 = [A_0 \cup T_{i_1}].$  Suppose that $\delta_p (Y_1) \ge 2\eta - (k+1)(k-1)^2.$  Since $\eta - (k+1)(k-1)^2 > 0$, it follows by Lemma \ref{obs6} that there is a $Y_1$-path from a vertex in $\overline{V}_W - Y_1$ to 
 to a vertex in $\tilde{Y}_1.$  Let 
$P_{i_2}$ be a such a $Y_1$-path from $y_{i_2}$ to $y_{i_2}'$ and let 
$a_{i_2}' = \psi(y_{i_2}').$  Let $Y_2 = [A_0 \cup T_{i_1} \cup T_{i_2}].$  If $\delta_p (Y_2) \ge 3\eta - (k+1)(k-1)^2 ,$ then given $\eta - (k+1)(k-1)^2  >0$,
Lemma \ref{obs6} implies that there is a $Y_2$-path from a vertex in $ \overline{V}_W - Y_2$ to a vertex in $\tilde{Y}_2.$   Let
Let $P_{i_3}$ be a $Y_2$-path from $y_{i_3}$ to $y_{i_3}'$. Let 
$a_{i_3}' = \psi(y_{i_3}').$   Since the process must eventually terminate, we may assume that for some $t>0$ we have generated $Y_1, \dots , Y_t$ 
and $\delta_p(Y_t) < (t+1)\eta - (k+1)(k-1)^2 .$   Let $\beta: [t] \rightarrow [k]$ be the function defined such that $\beta(j) := i_j, \ j \in [t].$  Let $K = (P_{\beta(1)}, \dots ,P_{\beta(t)})$ be the corresponding path-tangle (which is consistent with $\T$ and rooted at $Y_0$) where for all $j\in [t],\ y_{\beta(j)} \xrightarrow [W] {P_{\beta(j)}} y_{\beta(j)}'.$

Let $A^1 = \{ a_{\beta(i)} \in A^{\spadesuit}\ \big| \ i \in [t] \}$ and let $A^2 = \{ a_{\beta(i)} \in A-A^{\spadesuit}\ \big| \ i \in [t] \}$.  Since\\ $\delta_p(Y_t) < (t+1)\eta - (k+1)(k-1)^2,$ it follows by
Lemma \ref{lem-needed1} that $\displaystyle{A_0 \subseteq \bigcup_{a_{\beta(j)} \in A^1} T_{\beta(j)}}$.   We may assume that $t$ is minimal with respect to this property; that is,
if one excludes the path $P_{\beta(t)}$ from $K$, the the resulting path-tangle no longer has this property.  In particular, this means that $a_{\beta(t)} \in A^\spadesuit.$

 Let $K_1 = (P_{\beta_1(1)}, \dots ,P_{\beta_1(t_1)})$ be a path-chain in $K$ from $y_{\beta(t)}$ to $Y_0$ where $P_{\beta_1(1)}$ terminates at an element $y_{\beta_1(1)}' \in Y_0.$
 %
%
If $y_{\beta_1(1)}' \in [T_{\beta_1(t_1)}],$ then $K'=K_1$ is the desired
path-chain.  Thus we may assume that $y_{\beta_1(1)}' \not\in [T_{\beta_1(t_1)}].$  Since $\displaystyle{A_0 \subseteq \bigcup_{a_{\beta(j)} \in A^1} T_{\beta(j)}}$, there exists $i' \in [t]$
for which $a_{\beta(i')} \in A^{1}$ and $y_{\beta_1(1)}' \in [T_{\beta(i')}].$  Now let $K_2 = (P_{\beta_2(1)}, \dots ,P_{\beta_2(t_2)})$ be a path-chain in $K$ from $y_{\beta(i')}$ to $Y_0$ where $P_{\beta_2(1)}$
 terminates at an element $y_{\beta_2(1)}' \in Y_0.$
 If  $y_{\beta_2(1)}' \in [T_{\beta_2(t_2)}],$ then 
$K'= K_2$ is the desired path-chain.
 On the other hand, if $y_{\beta_2(1)}' \in [T_{\beta_1(t_1)}],$ then $K'= K_1 \circ K_2,$ the composition of $K_2$ with $K_1$,  is the desired path-chain.  Thus we may assume that $y_{\beta_2(1)}' \not\in [T_{\beta_1(t_1)}] \cup [T_{\beta_2(t_2)}].$  
 
 Continuing, there exists $i'' \in [t]$ for which $a_{\beta(i'')}\in A^{1}$ and $y_{\beta_2(1)}' \in [T_{\beta(i'')}].$
 Let $K_3 = (P_{\beta_3(1)}, \dots ,P_{\beta_3(t_3)})$ be a path-chain in $K$ from $y_{\beta(i'')}$ to $Y_0$ where $P_{\beta_3(1)}$ terminates at an element $y_{\beta_3(1)}' \in Y_0.$
If  $y_{\beta_3(1)}' \in [T_{\beta_3(t_3)}],$ then
$K' =K_3$
is the desired path-chain.  If $y_{\beta_3(1)}' \in [T_{\beta_2(t_2)}],$ then $K' =K_2 \circ K_3$ is the desired path-chain.  If $y_{\beta_3(1)}' \in [T_{\beta_1(t_1)}],$ then $K' = K_1 \circ K_2 \circ K_3$ is the desired path-chain.  Thus we may assume that 
$y_{\beta_3(1)}' \not\in \bigcup_{i=1}^3 [T_{\beta_i(t_i)}].$  Seeing as this process can not continue indefinitely, we must eventually arrive at the desired path-chain $K'$.
\end{proof}

\subsection{Proof of Theorem \ref{the-augment}}

By Proposition \ref{pro-pathchain}, there is a function $\gamma: [s] \rightarrow [k]$ for which $K' = (P_{\gamma(1)}, \dots ,P_{\gamma(s)})$
 is a path-chain where $a_{\gamma(s)} \in A^\spadesuit$ and $a_{\gamma(1)}' = \psi(y_{\gamma(1)}') \in A_0 \cap T_{\gamma(s)}.$
Among such path-chains,  we may assume $K'$ has a minimum number of paths. By minimality, the paths $P_1, \dots ,P_{s}$ are vertex-disjoint;  if two of these paths intersected, then Observation \ref{obs-alternating1} implies that one can reduce the number of paths in $K'.$  By minimality, we also see that for $j=2, \dots ,s,$ $a_{\gamma(j)}' = \psi(y_{\gamma(j)}')\in T_{\gamma(j-1)} \backslash \bigcup_{j' < j-1} T_{\xi(j')}.$  Thus $K'$ is a simple path-chain. The minimality of $K'$ also implies that for all $0 \le j-i \le s-2$ either $a_{\gamma(i)}' \not\in A_0\cap T_{\gamma(j)}$ or $a_{\gamma(j)} \not\in A^\spadesuit.$ 
It now follows by Proposition \ref{pro-SimpleTangle} that $A'= A - \{ a_{\gamma(1)}, \dots ,a_{\gamma(s)} \} + \{ a_{\gamma(1)}', \dots ,a_{\gamma(s)}' \}$ is a basis for $N.$  
Thus the matching $W' = W \triangle E(P_{\gamma(1)} \cup \cdots \cup P_{\gamma(s)})$ is such that $V_{W'}$ is a basis for $M$ (since $\psi(\overline{V}_{W'}) = A'$).   We claim that $W \preceq W'.$   Since $K'$ is consistent with $\T$ (an $F_0$-target), we have that for all $i \in [s-1],$ $a_{\gamma(i+1)}'\in T_{\gamma(i)}.$  Suppose $a_{\gamma(i)} \in A-A^{\spadesuit} .$  Then $T_{\gamma(i)} = F_{\gamma(i)}'$ and this means that $a_{\gamma(i+1)}' \in F_{\xi(i)}'.$  Thus when we replace $a_{\gamma(i)}$ by $a_{\gamma(i+1)}'$, the new basis element $a_{\gamma(i+1)}'$ will also belong to $F_{\gamma(i)}.$   Because of this, we ensure that $W \preceq W'.$  Furthermore, since we replace $a_{\gamma(s)} \in A^{\spadesuit}$ by 
$a_{\gamma(1)}' \in F_0 \cap T_{\gamma(s)},$ we have that $\omega_{W'}(F_0) > \omega_W(F_0)$.
   
Note that for each path $P_{\gamma(i)},$ $\delta_p(y_{\gamma(i)}') >0$.  As such, we have $\forall v \in \overline{V}_{W'}, \ \delta_p(v) >0.$
If $\omega_{W'}(F_0) < r_N(F_0)$, then we can continue finding a matching $W''$ for which $\omega_{W''}(F_0) > \omega_{W'}(F_0)$ and $W' \preceq W''.$  Continuing the above process, one can find a matching $W'''$ for which $W \preceq W'''$, $\omega_{W''}(F_0) = r_N(F_0).$   If there are flats $F\in \frakF$ for which $\omega_{W'''}(F) < r_N(F),$ then we can pick one such flat of smallest rank, say $F_0'$, and repeat the previous process with $F_0'$ in place of $F_0$, and $W'''$ in place of $W.$  
Repeating this procedure, we eventually arrive at a matching $W^*$ for which $\forall F\in \frakF,\ \omega_{W^*}(F) = r_N(F).$

\section{Proof of Theorem \ref{the-main}}

We may assume that we have removed matchings $W_1, \dots ,W_p$ from $G$ where for all $i\in [p],$ $V_{W_i}$ is a basis in $M.$  Assuming $\eta = n - p > (k+1)(k-1)^2,$ it follows by Theorem \ref{the-augment}, one can find a matching $W= W_{p+1}$ in $G_p$ for which $V_W$ is a basis, and  i) $\forall v\in V, \delta_{p+1}(v) \ge 0,$ and ii) $\forall F\in \frakF, \ \omega_{W} = r_N(F).$
  It follows that (\ref{eqn-Fineq}) holds for all flats $F$ in $N$ and this in turn implies that for all flats $F,$ $\delta_{p+1}(F) \le r_N(F) (\eta -1).$
By Lemma \ref{lem-augment}, the graph $G_{p+1} = G_p - W_{p+1}$ contains a matching $W = W_{p+2}$ for which $V_{W_{p+2}}$ is a basis of $M.$ If $\eta -1 > (k+1)(k-1)^2$, then the process can continue where one alters $W$ etc. as before.  Continuing, it is seen that one can construct at least $n - (k+1)(k-1)^2 -1 > n - k^3$ matchings $W_1, W_2, \dots $ for which $V_{W_i}$ is a basis of $M.$

\bibliographystyle{abbrv}
\bibliography{rota2}

\end{document}